\documentclass[12pt]{article}
\usepackage{amssymb, amsmath, latexsym, a4}
\usepackage[latin1]{inputenc}
\usepackage{graphicx}
\usepackage[all]{xy}
\usepackage[usenames]{color}

\newenvironment{proof}[1][Proof]{\noindent\textbf{#1.} }{\ \rule{0.5em}{0.5em}}

\newtheorem{De}{Definition}[section]
\newtheorem{Th}[De]{Theorem}
\newtheorem{Pro}[De]{Proposition}
\newtheorem{Le}[De]{Lemma}
\newtheorem{Co}[De]{Corollary}
\newtheorem{Rem}[De]{Remark}
\newtheorem{Ex}[De]{Example}

\newcommand{\Hom}{{\sf Hom}}
\newcommand{\Image}{{\sf Im}}
\newcommand{\Ker}{{\sf Ker}}
\newcommand{\Der}{{\sf Der}}

\newcommand{\id}{{\sf Id}}
\newcommand{\Lie}{\ensuremath{\mathsf{Lie}}}

\newcommand{\Lieh}{\ensuremath{\mathfrak{h}}}
\newcommand{\Lieg}{\ensuremath{\mathfrak{g}}}

\newcommand{\Lieq}{\ensuremath{\mathfrak{q}}}

\newcommand{\Liem}{\ensuremath{\mathfrak{m}}}

\newcommand{\LieL}{\ensuremath{\mathcal{L}}}

\newcommand{\Lien}{\ensuremath{\mathfrak{n}}}

\newcommand{\Leib}{\ensuremath{\mathsf{Leib}}}

\newcommand{\ze}{{\cal Z}}

\newbox\pullbackbox
\setbox\pullbackbox=\hbox{\xy 0;<1mm,0mm>: \POS(4,0)\ar@{-} (0,0) \ar@{-} (4,4)
\endxy}

\newdir{>>}{{}*!/3.5pt/:(1,-.2)@^{>}*!/3.5pt/:(1,+.2)@_{>}*!/7pt/:(1,-.2)@^{>}*!/7pt/:(1,+.2)@_{>}}
\newdir{ >>}{{}*!/8pt/@{|}*!/3.5pt/:(1,-.2)@^{>}*!/3.5pt/:(1,+.2)@_{>}}
\newdir{ |>}{{}*!/-3.5pt/@{|}*!/-8pt/:(1,-.2)@^{>}*!/-8pt/:(1,+.2)@_{>}}
\newdir{ >}{{}*!/-8pt/@{>}}
\newdir{>}{{}*:(1,-.2)@^{>}*:(1,+.2)@_{>}}
\newdir{<}{{}*:(1,+.2)@^{<}*:(1,-.2)@_{<}}

  \newcommand{\eh}{\frak h}

\begin{document}

\centerline{\bf  {\Lie}-central derivations, {\Lie}-centroids and {\Lie}-stem Leibniz algebras}

\bigskip
\centerline{\bf G. R. Biyogmam$^1$, J. M. Casas$^2$ and N. Pacheco Rego$^3$}

\bigskip
\centerline{$^1$Department of Mathematics, Georgia College \& State University}
\centerline{Campus Box 17 Milledgeville, GA 31061-0490}
\centerline{ {E-mail address}: guy.biyogmam@gcsu.edu}
\bigskip

\centerline{$^2$Dpto. Matemática Aplicada I, Universidade de Vigo,  E. E. Forestal}
\centerline{Campus Universitario A Xunqueira, 36005 Pontevedra, Spain}
\centerline{ {E-mail address}: jmcasas@uvigo.es}
\bigskip

\centerline{$^{3}$IPCA, Dpto. de Ciências, Campus do IPCA,
 Lugar do Aldão}
\centerline{4750-810 Vila Frescainha, S. Martinho, Barcelos,
 Portugal}
\centerline{E-mail address: \tt natarego@gmail.com}

\bigskip

\date{}

\bigskip \bigskip \bigskip

{\bf Abstract:}
In this paper, we introduce the notion  \Lie-derivation. This concept generalizes  derivations for non-$\Lie$ Leibniz algebras. We study these $\Lie$-derivations in the case where their image is contained in the $\Lie$-center, call them $\Lie$-central derivations.  We provide a characterization of \Lie-stem Leibniz algebras by their \Lie-central derivations, and prove several properties of the Lie algebra of {\Lie}-central derivations for {\Lie}-nilpotent Leibniz algebras of class 2. We also introduce  ${\sf ID}_*$-${\Lie}$-derivations.  A ${\sf ID}_*$-${\Lie}$-derivation of a Leibniz algebra $\Lieg$ is a $\Lie$-derivation of $\Lieg$ in which the image is contained in the second term of the lower {\Lie}-central series of $\Lieg$, and that vanishes on {\Lie}-central elements. We  provide an upperbound for the dimension of the Lie algebra ${\sf ID}_*^{\Lie}(\Lieg)$ of ${\sf ID}_*$-${\Lie}$-derivation of $\Lieg$,  and prove that the sets ${\sf ID}_*^{\Lie}(\Lieg)$ and ${\sf ID}_*^{\Lie}(\Lieq)$ are isomorphic for any two \Lie-isoclinic Leibniz algebras $\Lieg$ and $\Lieq.$
\bigskip

\bigskip



\section{Introduction}
Studies such as the work of Dixmier \cite{Dix}, Leger \cite{Leg} and T\^{o}g\^{o} \cite{Tog, Tog1, Tog2, Tog3} about the structure of  a  Lie algebra $\LieL$  and its relationship with the properties of the Lie algebra of derivations of  $\LieL$ have been conducted by several authors.
 A classical problem concerning the algebra of derivations consists to determine necessary and sufficient conditions under which subalgebras of the algebra of derivations coincide. For example, the coincidence of the subalgebra of central derivations with the algebra of derivations  of a Lie algebra was studied in \cite{Tog1}. Also centroids play important roles in the study of extended affine Lie algebras \cite{BN}, in the investigations of the Brauer groups and division algebras, in the classification of algebras or in the structure theory of algebras.  Almost inner derivations arise in many contexts of algebra, number theory or geometry, for instance they play an important role in the study of isospectral deformations of compact solvmanifolds \cite{GW}; the paper \cite{BDV} is dedicated to study almost inner derivations of Lie algebras.

Our aim in this paper is to conduct an analogue study  by investigating various concepts of derivations on Leibniz algebras. Our study relies on the relative notions of these derivations; derivations relative to the  \emph{Liezation functor} $(-)_{\Lie} : {\Leib} \to {\Lie}$, which assigns to a Leibniz algebra $\Lieg$ the Lie algebra ${\Lieg}_{_{\rm Lie}}$, where {\Leib} denotes the category of Leibniz algebras and {\Lie} denotes the category of Lie algebras.

The approached properties are closely related to the relative notions of central extension in a semi-abelian category with respect to a Birkhoff subcategory \cite{CVDL1, EVDL}. A recent research line  deals with the development of absolute properties of Leibniz algebras (absolute are the usual properties and it  means relative to the abelianization functor)  in the  relative setting (with respect to the Liezation functor); in general, absolute properties have the corresponding relative ones, but not all absolute properties immediately hold in the relative case, so new requirements are needed as it can be seen in the following papers \cite{BC, BC1, BC2, CI, CKh, RC}.

In order to develop a systematic study of derivation in the relative setting,
we organize the paper as follows: in Section \ref{preliminaries}, we provide some background on relative notions with respect to the Liezation functor.  We define the sets of \Lie-derivations $\Der^{\Lie}(\Lieg)$ and central \Lie-derivations $\Der^{\Lie}_z(\Lieg)$ for a non-Lie Leibniz algebra $\Lieg.$ It is worth mentioning that the absolute derivations are also {\Lie}-derivations. In Section \ref{stem}, we characterize {\Lie}-stem Leibniz algebras using  their {\Lie}-central derivations. Using {\Lie}-isoclinism, we prove several results on the Lie algebra of {\Lie}-central derivations of {\Lie}-nilpotent Leibniz algebras of class two.  In concrete, we prove that $\Der^{\Lie}_z(\Lieg)$ is abelian if and only if $Z_{\Lie}(\Lieg) = \gamma_2^{\Lie}(\Lieg)$, under the assumption that {\Lieg} is a finite-dimensional {\Lie}-nilpotent Leibniz algebra of class 2.  In Section \ref{centroids}, we define the \Lie-centroid  $\Gamma^{\Lie}(\Lieg)$ of $\Lieg$ and prove several of its basic properties. In particular we study its relationship with the \Lie-algebra $\Der^{\Lie}_z(\Lieg).$ In Section \ref{almost}, we study the set ${\sf ID}_*(\Lieg)$ of  ${\sf ID}_*$-${\Lie}$-derivations of a Leibniz algebra $\Lieg$ and its subalgebra   $\Der^{\Lie}_c(\Lieg)$ of almost inner {\Lie}-derivations of $\Lieg.$
Similarly to the result of T\^{o}g\^{o} \cite{Tog2} on derivations of Lie algebras, we provide necessary and sufficient conditions  on a finite dimensional Leibniz algebra $\Lieg$ for the subalgebras $\Der^{\Lie}_z(\Lieg)$ and ${\sf ID}_*(\Lieg)$ to be equal. We also prove that if two Leibniz algebras are \Lie-isoclinic, then their sets of ${\sf ID}_*$-${\Lie}$-derivations are isomorphic. This isomorphism also holds for their sets of  almost inner {\Lie}-derivations. We establish several results  on almost inner {\Lie}-derivations, similar  to the Lie algebra case \cite{BDV}. Finally, we provide an upperbound of the dimension of ${\sf ID}_*(\Lieg)$ by means of the dimension of $[\Lieg,\Lieg]_{\Lie}.$


\section{Preliminaries on Leibniz algebras} \label{preliminaries}

Let $\mathbb{K}$ be a fix ground field such that $\frac{1}{2} \in \mathbb{K}$. Throughout the paper, all vector spaces and tensor products are considered over $\mathbb{K}$.

A \emph{Leibniz algebra} \cite{Lo 1, Lo 2} is a vector space ${\Lieg}$  equipped with a bilinear map $[-,-] : \Lieg \otimes \Lieg \to \Lieg$, usually called the \emph{Leibniz bracket} of ${\Lieg}$,  satisfying the \emph{Leibniz identity}:
\[
 [x,[y,z]]= [[x,y],z]-[[x,z],y], \quad x, y, z \in \Lieg.
\]

 A subalgebra ${\eh}$ of a Leibniz algebra ${\Lieg}$ is said to be \emph{left (resp. right) ideal} of ${\Lieg}$ if $ [h,g]\in {\eh}$  (resp.  $ [g,h]\in {\eh}$), for all $h \in {\eh}$, $g \in {\Lieg}$. If ${\eh}$ is both
left and right ideal, then ${\eh}$ is called \emph{two-sided ideal} of ${\Lieg}$. In this case $\Lieg/\Lieh$ naturally inherits a Leibniz algebra structure.

Given a Leibniz algebra ${\Lieg}$, we denote by ${\Lieg}^{\rm ann}$ the subspace of ${\Lieg}$ spanned by all elements of the form $[x,x]$, $x \in \Lieg$. It is clear that the quotient ${\Lieg}_ {_{\rm Lie}}=\Lieg/{\Lieg}^{\rm ann}$ is a Lie algebra. This defines the so-called  \emph{Liezation functor} $(-)_{\Lie} : {\Leib} \to {\Lie}$, which assigns to a Leibniz algebra $\Lieg$ the Lie algebra ${\Lieg}_{_{\rm Lie}}$. Moreover, the canonical epimorphism  ${\Lieg} \twoheadrightarrow {\Lieg}_ {_{\rm Lie}}$ is universal among all homomorphisms from $\Lieg$ to a Lie algebra, implying that the Liezation functor is left adjoint to the inclusion functor $ {\Lie} \hookrightarrow {\Leib}$.

Given a Leibniz algebra $\Lieg,$ we define the bracket $$[-,-]_{lie}:\Lieg\to \Lieg, ~~ \text{by} ~~ [x,y]_{lie}=[x,y]+[y,x], ~~ \text{for} ~~ x,y\in\Lieg.$$

Let  ${\Liem}$, ${\Lien}$ be two-sided ideals of a Leibniz algebra  ${\Lieg}$. The following notions come from \cite{CKh}, which were derived from \cite{CVDL1}.

 The \emph{$\Lie$-commutator} of  ${\Liem}$ and  ${\Lien}$ is the two-sided ideal  of $\Lieg$
\[
[\Liem,\Lien]_{\Lie}= \langle \{[m,n]_{lie}, m \in \Liem, n \in \Lien \}\rangle.
\]

The \emph{$\Lie$-center} of the Leibniz algebra $\Lieg$ is the two-sided ideal
\[
Z_{\Lie}(\Lieg) =  \{ z\in \Lieg\,|\,\text{$[g,z]_{lie}=0$ for all $g\in \Lieg$}\}.
\]

 The \Lie-{\it centralizer} of ${\Liem}$ and ${\Lien}$ over  ${\Lieg}$ is
\[
C_{\Lieg}^{\Lie}({\Liem} , {\Lien}) = \{g \in {\Lieg} \mid  \; [g, m]_{lie} \in {\Lien}, \; \text{for all} \;
m \in {\Liem} \} \; .
\]

Obviously, $C_{\Lieg}^{\Lie}(\Lieg,0)=Z_{\Lie}(\Lieg)$.

 The right-center of a Leibniz algebra $\Lieg$ is the two-sided ideal $Z^r(\Lieg) = \{ x \in \Lieg\mid [y,x]=0 ~\text{for all}~y \in \Lieg \}$. The left-center of a Leibniz algebra $\Lieg$ is the set $Z^l(\Lieg) = \{ x \in \Lieg\mid [x,y]=0 ~\text{for all}~y \in \Lieg \}$, which might not  even be a subalgebra. $Z(\Lieg) = Z^l(\Lieg) \cap Z^r(\Lieg)$ is called the center of $\Lieg$, which is a two-sided ideal of $\Lieg$.

\begin{De} \cite{CKh}
Let ${\Lien}$ be a two-sided ideal of a Leibniz algebra $\Lieg$. The lower $\Lie$-central series of $\Lieg$ relative to ${\Lien}$ is the sequence
\[
\cdots \trianglelefteq {\gamma_i^{\Lie}(\Lieg,\Lien)} \trianglelefteq \cdots \trianglelefteq \gamma_2^{\Lie}(\Lieg,\Lien)  \trianglelefteq {\gamma_1^{\Lie}(\Lieg,\Lien)}
\]
of two-sided ideals of $\Lieg$ defined inductively by
\[
{\gamma_1^{\Lie}(\Lieg,\Lien)} = {\Lien} \quad \text{and} \quad \gamma_i^{\Lie}(\Lieg,\Lien) =[{\gamma_{i-1}^{\Lie}(\Lieg,\Lien)},{\Lieg}]_{\Lie}, \quad   i \geq 2.
\]
\end{De}

We use the notation $\gamma_i^{\Lie}(\Lieg)$ instead of $\gamma_i^{\Lie}(\Lieg,\Lieg), 1 \leq i \leq n$.

If $\varphi : \Lieg \to \Lieq$ is a homomorphism of Leibniz  such that $\varphi({\Liem}) \subseteq {\Lien}$, where ${\Liem}$ is a two-sided ideal of ${\Lieg}$ and ${\Lien}$ a two-sided ideal of ${\Lieq}$, then $\varphi(\gamma_{i}^{\Lie}({\Lieg}, {\Liem})) \subseteq \gamma_{i}^{\Lie}({\Lieq}, {\Lien}), i \geq 1$.

\begin{De}
The Leibniz algebra $\Lieg$ is said to be $\Lie$-nilpotent relative to $\Lien$ of class $c$ if\ $\gamma_{c+1}^{\Lie}(\Lieg, \Lien) = 0$ and $\gamma_c^{\Lie}(\Lieg, \Lien) \neq 0$.
\end{De}

\begin{De} \cite{CKh}
The upper $\Lie$-central series of a Leibniz algebra ${\Lieg}$ is the sequence of two-sided ideals, called $i$-{\Lie} centers, $i = 0, 1, 2,  \dots,$
 \[
{\ze}_0^{\Lie}({\Lieg}) \trianglelefteq {\ze}_1^{\Lie}({\Lieg}) \trianglelefteq \cdots \trianglelefteq {\ze}_i^{\Lie}({\Lieg}) \trianglelefteq \cdots
\]
 defined inductively by
\[
{\ze}_0^{\Lie}({\Lieg}) = 0 \quad \text{and} \quad
 {\ze}_{i}^{\Lie}({\Lieg}) = C_{\Lieg}^{\Lie}({\Lieg},{\ze}_{i-1}^{\Lie}({\Lieg})) , \  i \geq 1 .
 \]
\end{De}

\begin{Th} \cite[Theorem 4]{CKh} \label{2.4}
A Leibniz algebra {\Lieg} is  $\Lie$-nilpotent of class $c$ if and only if $Z_c^{\Lie}(\Lieg) ={\Lieg}$ and $Z_{c-1}^{\Lie}(\Lieg) \neq {\Lieg}$.
\end{Th}

\begin{De} \cite[Definition 2.8]{CI}
Let $\Liem$ be a subset of a Leibniz algebra $\Lieg$. The \Lie-normalizer of $\Liem$ is the subset of $\Lieg$:
$$N_{\Lieg}(\Liem) = \{ g \in \Lieg \mid [g,m], [m,g] \in \Liem, \text{\rm for all}\ m \in \Liem \}$$
\end{De}

\begin{De} \cite[Proposition 1]{CKh}
An exact sequence of Leibniz algebras $0 \to \Lien \to \Lieg \stackrel{\pi} \to \Lieq \to 0$ is said to be \emph{$\Lie$-central} extension if $[{\Lieg}, {\Lien}]_{\Lie}=0$, equivalently
$\Lien \subseteq Z_{\Lie}(\Lieg)$.
\end{De}

\begin{De} A linear map $d : \Lieg \rightarrow \Lieg$ of a Leibniz algebra $(\Lieg,[-,-])$ is said to be a {\Lie}-derivation if for all $x, y \in \Lieg$, the following  condition holds: $$d([x,y]_{lie})=[d(x),y]_{lie} + [x, d(y)]_{lie}$$
\end{De}

We denote by $\Der^{\Lie}(\Lieg)$ the set of all {\Lie}-derivations of a Leibniz algebra $\Lieg$, which can be equipped with a structure of Lie algebra by means of the usual bracket $[d_1, d_2] = d_1 \circ d_2 - d_2 \circ d_1$, for all $d_1, d_2 \in \Der(\Lieg)$.


\begin{Ex}
The absolute  derivations, that is  linear maps $d \colon \Lieg \to \Lieg$ such that $d([x,y])= [d(x),y]+[x,d(y)]$, are also {\Lie}-derivations, since:
\begin{equation} \label{eq1}
d([x,y]_{lie}) = d([x,y]+[y,x]) =[d(x), y]_{lie} + [x, d(y)]_{lie}, \quad \text{for all}~x, y \in \Lieg.
\end{equation}
In particular, for a fixed $x \in {\Lieg}$, the inner derivation $R_x \colon \Lieg \to \Lieg, R_x(y) = [y, x]$, for all $y \in {\Lieg}$, is a {\Lie}-derivation, so it gives rise to the following identity:
$$[[y,z]_{lie},x] = [[y,x],z]_{lie} + [y,[z,x]]_{lie}, \quad \text{for all}~x, y \in \Lieg.$$

However there are {\Lie}-derivations which are not derivations. For instance, every linear map $d \colon {\Lieg} \to {\Lieg}$ is a {\Lie}-derivation for any Lie algebra {\Lieg}, but it is not a derivation in general.
\end{Ex}

\begin{De} \label{Lie central der}
A {\Lie}-derivation $d : \Lieg \to \Lieg$ of a Leibniz algebra $\Lieg$ is said to be $\Lie$-central derivation if its image is contained in the $\Lie$-center of $\Lieg$.
\end{De}

\begin{Rem}
The absolute notion corresponding to Definition \ref{Lie central der} is the so called central derivations, that is derivations $d : \Lieg \to \Lieg$ such that its image is contained in the center of $\Lieg$. Obviously, every central derivation is a $\Lie$-central derivation. However the converse is not true as the following example shows:
let $\Lieg$ be the two-dimensional Leibniz algebra with basis $\{e, f\}$ and bracket operation given by $[e,f] = - [f, e] = e$ \cite{Cu}.  The inner derivation $R_e$ is a $\Lie$-central derivation, but it is not central in general.
\end{Rem}

We denote the set of all $\Lie$-central derivations of a Leibniz algebra $\Lieg$ by $\Der_z^{\Lie}(\Lieg)$. Obviously  $\Der_z^{\Lie}(\Lieg)$ is a subalgebra of $\Der^{\Lie}(\Lieg)$ and every element of $\Der_z^{\Lie}(\Lieg)$ annihilates $\gamma_2^{\Lie}(\Lieg) = [\Lieg, \Lieg]_{\Lie}$.   $\Der_z^{\Lie}(\Lieg)= C_{\Der^{\Lie}(\Lieg)}(({\sf R+L})(\Lieg) )$, where ${\sf L}(\Lieg)=\{ L_x \ | \ x\in \Lieg \}$, $L_x$ denotes the left multiplication operator $L_x(y) = [x,y]$, ${\sf R}(\Lieg)=\{ R_x \ | \ x\in \Lieg \}$ and $C_{\Lieg} (\Liem) = \{ x \in \Lieg \mid [x, y] = 0 = [y, x], \text{for all}~ y \in \Liem \}$, the absolute centralizer of an ideal $\Liem$ over the Leibniz algebra $\Lieg$.

Let A and B be two Leibniz algebras  and $T(A,B)$ denotes the set of all linear transformations from $A$ to $B$. Clearly, $T(A,B)$ endowed with the bracket $[f,g](x)=[f(x),g(x)]$ is an abelian Leibniz algebra if $B$ is an abelian Leibniz algebra too.

Consider the $\Lie$-central extensions $(g) : 0 \to \Lien \stackrel{\chi} \to \Lieg \stackrel{\pi} \to \Lieq \to 0$ and $(g_i) : 0 \to \Lien_i \stackrel{\chi_i}\to \Lieg_i \stackrel{\pi_i} \to \Lieq_i \to 0, i=1, 2.$ 

Let be $C : \Lieq \times \Lieq \to [\Lieg, \Lieg]_{\Lie}$ given by $C(q_1,q_2)=[g_1,g_2]_{lie}$, where $\pi(g_j)=q_j, j= 1, 2$, the $\Lie$-commutator map associated to the extension $(g)$. In a similar way are defined the $\Lie$-commutator maps $C_i$ corresponding to the extensions $(g_i), i = 1, 2$.

Note that if $\Lieq$ is a Lie algebra, then $\pi([\Lieg, \Lieg]_{\Lie})=0$, hence $[\Lieg, \Lieg]_{\Lie} \subseteq \Lien \equiv \chi(\Lien)$.

\begin{De} \label{isoclinic} \cite[Definition 3.1]{BC}
The $\Lie$-central extensions $(g_1)$ and $(g_2)$ are said to be \Lie-isoclinic when there exist isomorphisms $\eta : \Lieq_1 \to \Lieq_2$ and $\xi : [\Lieg_1, \Lieg_1]_{\Lie} \to [\Lieg_2, \Lieg_2]_{\Lie}$ such that the following diagram is commutative:
\begin{equation}  \label{square isoclinic}
\xymatrix{
\Lieq_1 \times \Lieq_1 \ar[r]^{C_1} \ar[d]_{\eta \times \eta} & [\Lieg_1, \Lieg_1]_{\Lie} \ar[d]^{\xi}\\
\Lieq_2 \times \Lieq_2 \ar[r]^{C_2} & [\Lieg_2, \Lieg_2]_{\Lie}
}
\end{equation}

The pair $(\eta, \xi)$ is called a \Lie-isoclinism from $(g_1)$ to $(g_2)$ and will be denoted by $(\eta, \xi) : (g_1) \to (g_2)$.
\end{De}

Let $\Lieg$ be a Leibniz algebra, then we can construct the following $\Lie$-central extension
\begin{equation} \label{Lie central extension}
(e_g) : 0 \to Z_{\Lie}(\Lieg) \to \Lieg \stackrel{pr_{\Lieg}} \to \Lieg/Z_{\Lie}(\Lieg) \to 0.
\end{equation}

\begin{De} \cite[Definition 3.3]{BC}
Let $\Lieg$ and $\Lieq$ be Leibniz algebras. Then $\Lieg$ and $\Lieq$ are said to be $\Lie$-isoclinic when $(e_g)$ and $(e_q)$ are $\Lie$-isoclinic $\Lie$-central extensions.

A $\Lie$-isoclinism $(\eta, \xi)$ from $(e_g)$ to $(e_q)$ is also called a $\Lie$-isoclinism  from $\Lieg$ to $\Lieq$, denoted by $(\eta, \xi) : \Lieg \sim \Lieq$.
\end{De}

\begin{Pro} \label{Lie-isoclinism} \cite[Proposition 3.4]{BC}
For a $\Lie$-isoclinism $(\eta, \xi) : (g_1) \sim (g_2)$, the following statements hold:
\begin{enumerate}
\item[a)] $\eta$ induces an isomorphism $\eta' : \Lieg_1/Z_{\Lie}(\Lieg_1) \to \Lieg_2/Z_{\Lie}(\Lieg_2)$, and $(\eta', \xi)$ is a $\Lie$-isoclinism from $\Lieg_1$ to $\Lieg_2$.
    \item[b)] $\chi_1(\Lien_1) = Z_{\Lie}(\Lieg_1)$ if and only if $\chi_2(\Lien_2) = Z_{\Lie}(\Lieg_2)$.
\end{enumerate}
\end{Pro}

\begin{De} \cite[Definition 4]{RC}
A $\Lie$-stem Leibniz algebra is a Leibniz algebra $\Lieg$ such that $Z_{\Lie}(\Lieg) \subseteq [\Lieg,\Lieg]_{\Lie}$.
\end{De}
Theorems 1 and 2 in \cite{RC} prove that every $\Lie$-isoclinic family of Leibniz algebras contains at least one $\Lie$-stem Leibniz algebra, which is of minimal dimension if it has finite dimension.


\section{$\Lie$-stem Leibniz algebras and $\Lie$-central derivations} \label{stem}

\begin{Pro} \label{3.1}
If $\Lieg$ is a $\Lie$-stem Leibniz algebra, then $\Der_z^{\Lie}(\Lieg)$ is an abelian Lie algebra.
\end{Pro}
\begin{proof}
Since  $\Der_z^{\Lie}(\Lieg)$ is a  subalgebra of  $\Der^{\Lie}(\Lieg),$ it is enough to show that $[d_1,d_2]=0$ for all $d_1, d_2\in  \Der_z^{\Lie}(\Lieg).$  First, we notice that if $d\in \Der_z^{\Lie}(\Lieg),$ then $d([x,y]_{lie})=0$  for all $x,y\in\Lieg$ since $d(x),d(y)\in Z_{\Lie}(\Lieg).$ So in particular, $d(Z_{\Lie}(\Lieg))=0$ since  $Z_{\Lie}(\Lieg) \subseteq [\Lieg,\Lieg]_{\Lie}$ as $\Lieg$ is a $\Lie$-stem Leibniz algebra. Now let $d_1, d_2 \in  \Der_z^{\Lie}(\Lieg)$ and  $x\in \Lieg.$ Then $d_1(x), d_2(x)\in Z_{\Lie}(\Lieg),$  which implies that  $[d_1, d_2](x)=d_1(d_2(x))-d_2(d_1(x))=0.$ Hence $[d_1, d_2]=0.$
\end{proof}
\medskip

The converse of the above result is not true in general. Indeed,  let $\Lieg$ be any  Lie algebra. Then $Z_{\Lie}(\Lieg)=\Lieg$ and so $\Der_z^{\Lie}(\Lieg)$ is an abelian Lie algebra. However $\Lieg$ is not a $\Lie$-stem Leibniz algebra since $Z_{\Lie}(\Lieg)=\Lieg\not\subseteq 0=[\Lieg,\Lieg]_{\Lie}.$

\begin{Pro} \label{3.2}
Let $\Lieg$ be a  $\Lie$-nilpotent finite dimensional Leibniz algebra such that $\gamma_2^{\Lie}(\Lieg) \neq 0$. Then $\Der_z^{\Lie}(\Lieg)$ is abelian if and only if $\Lieg$ is a $\Lie$-stem Leibniz algebra.
\end{Pro}
\begin{proof}
We only need to prove the converse of Proposition \ref{3.1}.
Assume that  $\Lieg$ is not a \Lie-stem Leibniz algebra. Then, there is some $z_1\in Z_{\Lie}(\Lieg)$ such that $z_1 \notin [\Lieg,\Lieg]_{\Lie}$. Since $\Lieg$ is a  $\Lie$-nilpotent Leibniz algebra and $\gamma_2^{\Lie}(\Lieg) \neq 0$, it follows that $Z_{\Lie}(\Lieg)\cap[\Lieg,\Lieg]_{\Lie}\neq 0.$ Let $z_2\in Z_{\Lie}(\Lieg)\cap[\Lieg,\Lieg]_{\Lie}, ~z_2\neq0,$ and consider the following maps:$$d_1:\Lieg\to\Lieg,~~d_1(z)=\begin{cases} z_1 ~~~\mbox{if}~~z=z_1\\ 0~~~~\mbox{if}~~z \neq z_1\end{cases}$$ and $$d_2:\Lieg\to\Lieg,~~d_2(z)=\begin{cases} z_2 ~~~\mbox{if}~~z=z_1\\ 0 ~~~~\mbox{if}~~z\neq z_1.\end{cases}$$
Clearly, $d_1$ and $d_2$ are \Lie-central derivations, and $d_1$ and $d_2$ do not commute, since $[d_1,d_2](z_1)=d_1(d_2(z_1))-d_2(d_1(z_1))=-z_2\neq 0.$ Therefore  $\Der_z^{\Lie}(\Lieg)$ is not abelian. This completes the proof.
\end{proof}

\begin{Le} \label{3.3}
Let $(\eta, \xi)$ be a $\Lie$-isoclinism between the Leibniz algebras $\Lieg$ and $\Lieq$. If $\Lieg$ is a $\Lie$-stem Leibniz algebra, then $\xi$ maps $Z_{\Lie}(\Lieg)$ onto $Z_{\Lie}(\Lieq) \cap [\Lieq, \Lieq]_{\Lie}$.
\end{Le}
\begin{proof}
Since $Z_{\mathsf{Lie}}(\mathfrak{g})\subseteq \lbrack \mathfrak{g},\mathfrak{g}]_{\mathsf{Lie}}$, then an element $z$ of $Z_{\mathsf{Lie}}(\mathfrak{g})$ can be written as
$z=\underset{i=1}{\overset{n}{\sum}}\lambda _{i}\left[ x_{i},y_{i}\right]_{lie}$, with $\lambda _{i} \in  \mathbb{K}$ and $x_{i},y_{i}\in \Lieg, i =1, \dots, n$.

Let $\eta' : {\Lieg}/Z_{\Lie}(\Lieg)\longrightarrow \mathfrak{q}/Z_{\mathsf{Lie}}(\mathfrak{q})$, $\eta' \left( x_{i}+Z_{\mathsf{Lie}}({\Lieg})\right) = \eta(x_{i})+Z_{\mathsf{Lie}}({\Lieq})$ and $\eta' \left( y_{i}+ \right.$ $\left. Z_{\mathsf{Lie}}(\Lieg)\right) =\eta(y_{i})+Z_{\mathsf{Lie}}(\mathfrak{q}), i = 1, \dots, n,$ the isomorphism provided by \cite[Proposition 3.4]{BC}. Then
 \[
 \begin{array}{lcl}
\xi \left( z\right) +Z_{\mathsf{Lie}}(\mathfrak{q}) &= &\xi \left( \underset{i=1}{\overset{n}{\sum }}\lambda _{i}\left[ x_{i},y_{i}\right]_{lie} \right) +Z_{\mathsf{Lie}}(\mathfrak{q)}\\
  & = &\underset{i=1}{\overset{n}{\sum }}\lambda _{i} \xi \left[ x_{i}, y_{i} \right]_{lie} +Z_{\mathsf{Lie}}(\mathfrak{q)}\\
& = &\underset{i=1}{\overset{n}{\sum }}\lambda _{i} \left[ \eta(x_i), \eta(y_i)\right]_{lie} +Z_{\mathsf{Lie}}(\mathfrak{q)}\\
& = &\eta' \left( \underset{i=1}{\overset{n}{\sum }} \lambda _{i}\left[ x_{i},y_{i}\right]_{lie} +Z_{\mathsf{Lie}}(\mathfrak{g)}\right)\\
&= & Z_{\mathsf{Lie}}(\mathfrak{q)}.
\end{array}
\]
The surjective character can be easily established.
\end{proof}

\begin{Pro} \label{mono}
Let $\Lieg$ and  $\Lieq$ be two $\Lie$-isoclinic Leibniz algebras and $\Lieg$ be a $\Lie$-stem Leibniz algebra. Then every $d \in \Der_z^{\Lie}(\Lieg)$ induces a $\Lie$-central derivation $d^{*}$ of $\Lieq$. Moreover, the map $d \mapsto d^*$ is a monomorphism from $\Der_z^{\Lie}(\Lieg)$ into $\Der_z^{\Lie}(\Lieq)$.
\end{Pro}
\begin{proof}
Let $(\eta,\xi)$ be a \Lie-isoclinism between $\Lieg$ and $\Lieq,$ and let $d\in \Der_z^{\Lie}(\Lieg).$ Then for any $y\in \Lieq,$ we have $y+Z_{\Lie}(\Lieq)=\eta(x+Z_{\Lie}(\Lieg))$ for some $x\in \Lieg$, since $\eta$ is bijective. Now consider the map $d^*:\Lieq\to\Lieq$ defined by $d^*(y)=\xi(d(x)),$  which  is well-defined  since $d(Z_{\Lie}(\Lieg))=0$ as  $Z_{\Lie}(\Lieg) \subseteq [\Lieg,\Lieg]_{\Lie}.$   Moreover, $d^*\in \Der_z^{\Lie}(\Lieq)$ since $d(x)\in Z_{\Lie}(\Lieg)$ and  $\xi(d(x))\in Z_{\Lie}(\Lieq)\cap [\Lieq,\Lieq]_{\Lie}$  by Lemma \ref{3.3}.
$d^*$ is a \Lie-derivation  since  $d^*([y_1,y_2]_{lie})=\xi(d([x_1,x_2]_{lie}))=\xi([d(x_1),x_2]_{lie}+[x_1,d(x_2)]_{lie})=\xi(0+0)=0$ and $[y_1,d^*(y_2)]_{lie}+ [d^*(y_1),y_2]_{lie}=0$ since $d^*(y_1), d^*(y_2)\in  Z_{\Lie}(\Lieq).$

 Clearly,  the map $\phi: d\to d^*$ is  linear and one-to-one  since $\xi$ an isomorphism. To show that $\phi$ is compatible with the Lie-bracket, let $d_1,d_2\in  \Der_z^{\Lie}(\Lieg).$ Then for $i,j=1,2,$ $d_i(\Lieg)\subseteq Z_{\Lie}(\Lieg)\subseteq [\Lieg,\Lieg]_{\Lie}$ and $d_j([\Lieg,\Lieg]_{\Lie})=0.$ So on one hand $[d_1,d_2]=d_1d_2-d_2d_1=0$ and thus  $[d_1,d_2]^*=0$ as $\xi$ is an isomorphism. On the other hand, $d^*_i(\Lieq)\subseteq Z_{\Lie}(\Lieq)\cap [\Lieq,\Lieq]_{\Lie}.$ So $d^*_j(d^*_i(\Lieq))=0,$ by definition of $d^*_j,$  and thus $[d^*_1,d^*_2]=d^*_1d^*_2-d^*_2d^*_1=0.$ Therefore $\phi([d_1,d_2])=[\phi(d_1),\phi(d_2)].$ This completes the proof.
\end{proof}

\begin{Le}\label{DT}
For any $\Lie$-stem Leibniz algebra $\Lieg$, there is a Lie algebra isomorphism  $\Der_z^{\Lie}(\Lieg)\cong T\left(\frac{\Lieg}{[\Lieg, \Lieg]_{\Lie}}, Z_{\Lie}(\Lieg)\right).$
 \end{Le}
\begin{proof}
Let $d\in \Der_z^{\Lie}(\Lieg)$ be,  then $d(\Lieg)\subseteq  Z_{\Lie}(\Lieg),$ and thus $d([\Lieg, \Lieg]_{\Lie})=0.$ So $d$ induces  the map  $\frac{\Lieg}{[\Lieg, \Lieg]_{\Lie}}\stackrel{\alpha_d}\longrightarrow Z_{\Lie}(\Lieg)$  defined by $\alpha_d(x+[\Lieg, \Lieg]_{\Lie})= d(x)$. Now define the map $\beta \colon \Der_z^{\Lie}(\Lieg)\longrightarrow T\left(\frac{\Lieg}{[\Lieg, \Lieg]_{\Lie}}, Z_{\Lie}(\Lieg)\right)$ by $\beta(d)=\alpha_d.$ Clearly, $\beta$ is a linear map, which  is one-to-one by definition of $\alpha_d.$

$\beta$ is onto since for a given $d^*\in T\left(\frac{\Lieg}{[\Lieg, \Lieg]_{\Lie}}, Z_{\Lie}(\Lieg)\right)$, there exists a linear map $d : \Lieg \to Z_{\Lie}(g)$, $d= d^* \circ \pi$, where $\pi : \Lieg \to \frac{\Lieg}{[\Lieg, \Lieg]_{\Lie}}$ is the canonical projection, such that $\beta(d) = d^*$. Finally, $d \in {\Der}_z^{\Lie}(\Lieg)$ since $d([x,y]_{lie})= d^*([\pi(x), \pi(y)]_{lie}) = d^*(\overline{0})  = 0$; on the other hand, $[d(x),y]_{lie}+[x,d(y)]_{lie}= [d^*(\pi(x)),y]_{lie}+[x,d^*(\pi(y))]_{lie} = 0$, since $d^*(\pi(x)), d^*(\pi(y)) \in Z_{\Lie}(\Lieg)$.  To finish, we show that $\beta([d_1,d_2]) = [\beta(d_1) ,\beta(d_2) ]$ for all $d_1,d_2\in~ \Der_z^{\Lie}(\Lieg).$ Indeed, let $x\in\Lieg.$  It is clear that   $~\beta([d_1,d_2])(\pi(x))=\alpha_{[d_1,d_2]}(\pi(x))=[d_1,d_2](x)=0$ since $d_1(\Lieg), d_2(\Lieg)\subseteq  Z_{\Lie}(\Lieg)\subseteq[\Lieg, \Lieg]_{\Lie}$   and $d_1( [\Lieg, \Lieg]_{\Lie})=d_2([\Lieg, \Lieg]_{\Lie})=0.$
On the other hand, $[\beta(d_1) ,\beta(d_2) ](\pi(x))=[\alpha_{d_1},\alpha_{d_2}](\pi(x))=\alpha_{d_1}(d_2(x))-\alpha_{d_2}(d_1(x))=0$  since  $\alpha_{d_1}([\Lieg, \Lieg]_{\Lie})=0=\alpha_{d_2}([\Lieg, \Lieg]_{\Lie}). $ Hence $\beta([d_1,d_2]) = [\beta(d_1) ,\beta(d_2) ].$ This completes the proof.
\end{proof}

\begin{Co} \label{TZ}
For any arbitrary Leibniz algebra $\Lieq$, the Lie algebra $\Der_z^{\Lie}(\Lieq)$ has a central  subalgebra $\Lien$ isomorphic to $T\left(\frac{\Lieg}{[\Lieg, \Lieg]_{\Lie}}, Z_{\Lie}(\Lieg)\right)$ for some $\Lie$-stem Leibniz algebra $\Lieg$ $\Lie$-isoclinic to $\Lieq$. Moreover, each element of $\Lien$ sends $Z_{\Lie}(\Lieq)$ to the zero subalgebra.
\end{Co}
\begin{proof}
By  \cite[Corollary 4.1]{BC2}, there is a $\Lie$-stem Leibniz algebra $\Lieg$ \Lie-isoclinic to $\Lieq.$  Denote this \Lie-isoclinism by $(\eta,\xi).$ Now, by the proof of Proposition \ref{mono}, $ \Lien:=\{d^*~|~d\in \Der_z^{\Lie}(\Lieg)\}$  is a subalgebra of $\Der_z^{\Lie}(\Lieq)$ isomorphic to $\Der_z^{\Lie}(\Lieg).$  $\Lien$ is a central subalgebra of $\Der_z^{\Lie}(\Lieq).$  Indeed, let $d_0\in \Lien$ and $d_1\in \Der_z^{\Lie}(\Lieq).$ Then for any $y\in \Lieq,$ we have by definition, $d_0^*(y)=\xi(d_0(x))$ with $\pi_2(y)=\eta(\pi_1(x)).$ So $d_1(d_0^*(y))=0$ since $d_0^*(\Lieq)\subseteq Z_{\Lie}(\Lieq)\cap [\Lieq,\Lieq]_{\Lie}$ by Lemma 3.3, and $d_1([\Lieq,\Lieq]_{\Lie})=0.$ Also, $d_0^*(Z_{\Lie}(\Lieq))=0$ since $\eta$ is one-to-one and $\xi$ is a homomorphism. In particular, $d_0^*(d_1(y))=0$ since $d_1(\Lieq)\subseteq Z_{\Lie}(\Lieq).$ Therefore $[d_0^*,d_1]=0.$ Moreover, for any $d_0^*\in \Lien,$ $d_0^*(Z_{\Lie}(\Lieq))=0$ as mentioned above. To complete the proof, notice that $\Der_z^{\Lie}(\Lieg)\cong T\left(\frac{\Lieg}{[\Lieg, \Lieg]_{\Lie}}, Z_{\Lie}(\Lieg)\right)$ thanks to Lemma \ref{DT}.
\end{proof}

\begin{Le} \label{gNilq}
Let $\Lieg$ and $\Lieq$  be  two  $\Lie$-isoclinic Leibniz algebras. If  $\Lieg$ is $\Lie$-nilpotent of class $c$, then so is  $\Lieq$.
\end{Le}
\begin{proof}
Notice that for all $g\in\Lieg$ and $x_1,x_2,\ldots,x_i\in\Lieg,$ and setting $\bar{t}:=t+Z_{\Lie}(\Lieg),$ $ ~t=g,x_1,x_2,\ldots, x_i,$ we have   $$[[[\bar{g},\bar{x}_1]_{lie},\bar{x}_2]_{lie},\ldots,\bar{x}_i]_{lie}=[[[g,{x}_1]_{lie},{x}_2]_{lie},\ldots,{x}_i]_{lie}+Z_{\Lie}(\Lieg).$$
  So $g\in {\ze}_{i+1}^{\Lie}(\Lieg)$ $\iff$ $g+Z_{\Lie}(\Lieg)\in {\ze}_{i}^{\Lie}(\Lieg/Z_{\Lie}(\Lieg)).$ Therefore ${\ze}_{i+1}^{\Lie}(\Lieg)/Z_{\Lie}(\Lieg)={\ze}_{i}^{\Lie}(\Lieg/Z_{\Lie}(\Lieg)).$    If $(\eta,\xi)$ is the \Lie-isoclinism between $\Lieg$ and $\Lieq,$ we have as  $\eta$ is an isomorphism,$$\eta({\ze}_{i+1}^{\Lie}(\Lieg)/Z_{\Lie}(\Lieg))=\eta({\ze}_{i}^{\Lie}(\Lieg/Z_{\Lie}(\Lieg)))={\ze}_{i}^{\Lie}(\Lieq/Z_{\Lie}(\Lieq)).$$ It follows that $$\Lieg/{\ze}_{i+1}^{\Lie}(\Lieg)\cong\frac{\Lieg/Z_{\Lie}(\Lieg)}{{\ze}_{i+1}^{\Lie}(\Lieg)/Z_{\Lie}(\Lieg)}\cong\frac{\Lieq/Z_{\Lie}(\Lieq)}{{\ze}_{i+1}^{\Lie}(\Lieq)/Z_{\Lie}(\Lieq)}\cong \Lieq/{\ze}_{i+1}^{\Lie}(\Lieq).$$
Now, assume that $\Lieg$ is \Lie-nilpotent of class $c.$ Then ${\ze}_{c}^{\Lie}(\Lieg)=\Lieg.$ So $\Lieq/{\ze}_{c}^{\Lie}(\Lieq)\cong \Lieg/{\ze}_{c}^{\Lie}(\Lieg)=0,$ implying that ${\ze}_{c}^{\Lie}(\Lieq)=\Lieq.$ Also $\Lieg/{\ze}_{c-1}^{\Lie}(\Lieg)\neq 0$ $\iff$ $\Lieq/{\ze}_{c-1}^{\Lie}(\Lieq)\neq 0.$ Hence $\Lieq$ is also Lie-nilpotent of class $c.$
\end{proof}

\begin{Co} \label{DT2}
Let $\Lieq$ be a $\Lie$-nilpotent Leibniz algebra of class 2. Then $\Der_z^{\Lie}(\Lieq)$ has a central subalgebra isomorphic to  $T\left(\frac{\Lieq}{Z_{\Lie}(\Lieq)}, [\Lieq, \Lieq]_{\Lie}\right)$ containing $(R+L)(\Lieq)$.
\end{Co}
\begin{proof}
 By  \cite[Corollary 4.1]{BC2}, there is a $\Lie$-stem Leibniz algebra $\Lieg$ \Lie-isoclinic to $\Lieq.$  Denote this \Lie-isoclinism by $(\eta,\xi).$ Since $\Lieq$ is $\Lie$-nilpotent Leibniz algebra of class 2, so is $\Lieg,$ thanks to Lemma \ref{gNilq}. Then $ Z_{\Lie}(\Lieg)=[\Lieg,\Lieg]_{\Lie}\stackrel{\xi}\cong  [\Lieq,\Lieq]_{\Lie},$ and $\frac{\Lieg}{[\Lieg,\Lieg]_{\Lie}}\cong\frac{\Lieg}{Z_{\Lie}(\Lieg)}\stackrel{\eta}\cong\frac{\Lieq}{Z_{\Lie}(\Lieq)}.$
 So $T\left(\frac{\Lieg}{  [\Lieg,\Lieg]_{\Lie}},  Z_{\Lie}(\Lieg)\right)\cong T\left(\frac{\Lieq}{Z_{\Lie}(\Lieq)}, [\Lieq, \Lieq]_{\Lie}\right).$ Therefore $\Der_z^{\Lie}(\Lieq)$ has a central subalgebra $\Lien$ isomorphic to  $T\left(\frac{\Lieq}{Z_{\Lie}(\Lieq)}, [\Lieq, \Lieq]_{\Lie}\right)$,  thanks to Corollary \ref{TZ}. Moreover, the map $\zeta: \frac{\Lieq}{Z_{\Lie}(\Lieq)}\to T\left(\frac{\Lieq}{Z_{\Lie}(\Lieq)}, [\Lieq, \Lieq]_{\Lie}\right)$ defined by $x+Z_{\Lie}(\Lieq)\mapsto \zeta(x+Z_{\Lie}(\Lieq)):\frac{\Lieq}{Z_{\Lie}(\Lieq)}\to [\Lieq, \Lieq]_{\Lie}$ with $\zeta(x+Z_{\Lie}(\Lieq))(y+Z_{\Lie}(\Lieq))=[x,y]_{lie},$ is a well-defined one-to-one linear map, since for all $x,x'\in \Lieq,$
 $$\begin{aligned}
 x-x'\in Z_{\Lie}(\Lieq)& \iff [x-x',y]_{lie}=0~~\mbox{for all}~y\in\Lieq \\&\iff[x,y]_{lie}=[x',y]_{lie}~~\mbox{for all}~y\in\Lieq \\& \iff \zeta(\overline{x})(y+Z_{\Lie}(\Lieq))=\zeta(\overline{x'})(y+Z_{\Lie}(\Lieq))~~\mbox{for all}~y\in\Lieq \\&\iff \zeta(\overline{x})=\zeta(\overline{x'}).
     \end{aligned}$$
Here we are using the notation $\overline{x} = x + Z_{\Lie}(\Lieq)$.

 $(R+L)(\Lieq) = \Image(\zeta) \subseteq T\left(\frac{\Lieq}{Z_{\Lie}(\Lieq)}, [\Lieq, \Lieq]_{\Lie}\right)$ since $\zeta(\overline{x})(\overline{y}) = [x,y]_{lie} = [x, y] + [y,x] = L_x(y)+ R_x(y)$.
\end{proof}
\bigskip

For any Leibniz algebra $\Lieg$  with $\gamma_2^{\Lie}(\Lieg)$ abelian, we denote  $$K(\Lieg) := \bigcap \Ker \left( f : \Lieg \to \gamma_2^{\Lie}(\Lieg) \right)$$

\begin{Le} \label{gammaK}
Let $\Lieq$  be  a $\Lie$-nilpotent Leibniz algebra of class 2. Then $\gamma_2^{\Lie}(\Lieq) = K(\Lieq)$.
\end{Le}
\begin{proof}
 Let $f : \Lieq \to \gamma_2^{\Lie}(\Lieq) $ be a  homomorphism of Leibniz algebras. Then for all $q_1,q_2\in\Lieq,$ $f([q_1,q_2]_{lie})= [f(q_1),f(q_2)]_{lie}\in[\gamma_2^{\Lie}(\Lieq),\gamma_2^{\Lie}(\Lieq)]_{\Lie}\subseteq  \gamma_3^{\Lie}(\Lieq) =0$ as $\Lieq$ is $\Lie$-nilpotent of class 2. So $\gamma_2^{\Lie}(\Lieq) \subseteq \Ker(f).$ Therefore $\gamma_2^{\Lie}(\Lieq) \subseteq K(\Lieq)$ since $f$ is arbitrary.

 For the reverse inclusion, we proceed by contradiction. Let $x\in K(\Lieq)$ such that $x\notin \gamma_2^{\Lie}(\Lieq),$ and let $\Lieh$ be the complement of $\left\langle x\right\rangle$ in $\Lieq.$ Then $\Lieh$ is a maximal subalgebra of $\Lieq.$ So  either $\Lieh+\gamma_2^{\Lie}(\Lieq) =\Lieh$ or $\Lieh+\gamma_2^{\Lie}(\Lieq) =\Lieq.$ The latter is not possible. Indeed, if $\Lieh+\gamma_2^{\Lie}(\Lieq) =\Lieq$ then  $\gamma_2^{\Lie}(\Lieq)=\gamma_2^{\Lie}(\Lieh+\gamma_2^{\Lie}(\Lieq) )\subseteq \gamma_2^{\Lie}(\Lieh)+\gamma_3^{\Lie}(\Lieq).$ But since $\Lieq$  is  a $\Lie$-nilpotent Leibniz algebra of class 2,  then $\gamma_3^{\Lie}(\Lieq)=0,$ which implies that $\gamma_2^{\Lie}(\Lieq)=\gamma_2^{\Lie}(\Lieh),$ and thus $\Lieq=\Lieh+\gamma_2^{\Lie}(\Lieq) =\Lieh+\gamma_2^{\Lie}(\Lieh)=\Lieh.$ A contradiction. So we have $\Lieh+\gamma_2^{\Lie}(\Lieq)=\Lieh,$  and thus $\gamma_2^{\Lie}(\Lieq)\subseteq \Lieh,$ which implies that  $\Lieh$ is a two-sided ideal of $\Lieq.$  Now, choose $q_0\in \gamma_2^{\Lie}(\Lieq) $ and consider the mapping $ f : \Lieq \to \gamma_2^{\Lie}(\Lieq)$ defined by $h+\alpha x\mapsto \alpha q_0.$ Clearly, $f$ is a well-defined homomorphism of Leibniz algebras, and $\Ker(f)=\Lieh.$ This is a contradiction since $x\in K(\Lieq)$ and $x\notin\Lieh.$ Thus  $K(\Lieq) \subseteq \gamma_2^{\Lie}(\Lieq) .$
\end{proof}

\begin{Th}\label{3.10}
Let $\Lieq$  be  a $\Lie$-nilpotent Leibniz algebra of class 2. Then $$Z\left( \Der_z ^{\Lie} (\Lieq) \right) \cong T\left(\frac{\Lieq}{Z_{\Lie}(\Lieq)}, [\Lieq, \Lieq]_{\Lie}\right).$$
\end{Th}
\begin{proof}
 By the proof of Corollary \ref{DT2}, $\Der_ z^{\Lie}(\Lieq)$ has a central subalgebra  $\Lien$ isomorphic to  $T\left(\frac{\Lieq}{Z_{\Lie}(\Lieq)}, [\Lieq, \Lieq]_{\Lie}\right),$ where  $ \Lien:=\{d^*~|~d\in \Der_z^{\Lie}(\Lieg)\}$ for some $\Lie$-stem Leibniz algebra $\Lieg$ \Lie-isoclinic to $\Lieq.$  Denote this \Lie-isoclinism by $(\eta,\xi).$

It remains to show that $Z\left( \Der_z^{\Lie} (\Lieq) \right)\subseteq\Lien,$ that is, given $T\in Z\left( \Der_ z^{\Lie} (\Lieq) \right),$ we must find $d\in \Der_z^{\Lie}(\Lieg)$ such that $T=d^*.$

 First, we claim that $T(\Lieq)\subseteq K(\Lieq).$ Indeed,  let
$f : \Lieq \to [\Lieq, \Lieq]_{\Lie}$ be a  homomorphism of Leibniz algebras. Then consider the mapping $t_f:\Lieq\to\Lieq$ defined by $t_f(z)=f(z).$ Clearly, $t_f\in \Der_z^{\Lie} (\Lieq) $ since  $t_f(\Lieq)\subseteq[\Lieq,\Lieq]_{\Lie}=Z_{\Lie}(\Lieq)$ as $\Lieq$  is  a $\Lie$-nilpotent Leibniz algebra of class 2. So $[T,t_f]=0$ and thus $f(T(z))=t_f(T(z))=T(t_f(z))=0$ for all $z\in\Lieq$ since $t_f(z)\in [\Lieq,\Lieq]_{\Lie}$ and  $T([\Lieq,\Lieq]_{\Lie})=0$ as $T\in  \Der_z^{\Lie} (\Lieq).$ Therefore $T(\Lieq)\subseteq\Ker(f).$ Hence $T(\Lieq)\subseteq K(\Lieq)$ since $f$ is arbitrary, which proves the claim.

It follows by Lemma \ref{gammaK} that $T(\Lieq)\subseteq [\Lieq,\Lieq]_{\Lie}.$ Now, for any $x\in \Lieg,$ we have $x+Z_{\Lie}(\Lieg)=\eta^{-1}(y+Z_{\Lie}(\Lieq))$ for some $y\in \Lieq$, since $\eta$ is bijective. Consider the map $d: \Lieg\to\Lieg$ defined by $x\mapsto \xi^{-1}(T(y)).$ Clearly $d$ is well-defined, and it is easy to show that $d\in \Der_z^{\Lie}(\Lieg)$ since $T(\Lieq)\subseteq [\Lieq,\Lieq]_{\Lie}=Z_{\Lie}(\Lieq).$
Hence $T=d^*.$ This completes the proof.
\end{proof}

\begin{Co} \label{abelian der}
Let $\Lieq$ be a  finite-dimensional $\Lie$-nilpotent   Leibniz algebra of class 2. Then $\Der_z^{\Lie}(\Lieq)$ is abelian if and only if   $\gamma_2^{\Lie}(\Lieq) = Z_{\Lie}(\Lieq)$.
\end{Co}
\begin{proof}
 Assume that  $\gamma_2^{\Lie}(\Lieq) = Z_{\Lie}(\Lieq),$ then by Proposition \ref{3.2},  $\Der_z^{\Lie}(\Lieq)$ is abelian since $\Lieq$ is a \Lie-stem Leibniz algebra.
Conversely, suppose that
$\Der_z^{\Lie}(\Lieq)$ is an abelian Lie algebra. Then, again by  Proposition \ref{3.2}, $\Lieq$ is a \Lie-stem Leibniz algebra. This implies by Lemma \ref{DT} that $\Der_z^{\Lie}(\Lieq)\cong T\left(\frac{\Lieq}{\gamma_2^{\Lie}(\Lieq)}, Z_{\Lie}(\Lieq)\right).$  Also, by Theorem  \ref{3.10}, $\Der_z^{\Lie} (\Lieq) =Z\left( \Der_z^{\Lie} (\Lieq) \right) \cong T\left(\frac{\Lieq}{Z_{\Lie}(\Lieq)}, \gamma_2^{\Lie}(\Lieq)\right).$ It follows that    $T\left(\frac{\Lieq}{\gamma_2^{\Lie}(\Lieq)}, Z_{\Lie}(\Lieq)\right)\cong T\left(\frac{\Lieq}{Z_{\Lie}(\Lieq)}, \gamma_2^{\Lie}(\Lieq)\right).$  Now, let $K$ be the $\mathbb{K}$-vector subspace complement of   $Z_{\Lie}(\Lieq)$ in $\gamma_2^{\Lie}(\Lieq)$.  We claim that $K=0.$ Indeed, since as vector spaces $ Z_{\Lie}(\Lieq) \oplus K =\gamma_2^{\Lie}(\Lieq)$, then $$T\left(\frac{\Lieq}{Z_{\Lie}(\Lieq)}, \gamma_2^{\Lie}(\Lieq) \right)=T\left(\frac{\Lieq}{Z_{\Lie}(\Lieq)}, Z_{\Lie}(\Lieq) \right)\oplus T\left(\frac{\Lieq}{Z_{\Lie}(\Lieq)}, K\right).$$ As  $\frac{\Lieq}{Z_{\Lie}(\Lieq)} \twoheadrightarrow \frac{\Lieq}{\gamma_2^{\Lie}(\Lieq)}$ by the Snake Lemma, it follows that $T\left(\frac{\Lieq}{Z_{\Lie}(\Lieq)}, \gamma_2^{\Lie}(\Lieq)\right) \cong T\left(\frac{\Lieq}{\gamma_2^{\Lie}(\Lieq)}, Z_{\Lie}(\Lieq)\right)$  is isomorphic to a subalgebra of $T\left(\frac{\Lieq}{Z_{\Lie}(\Lieq)}, Z_{\Lie}(\Lieq) \right).$ Hence $T\left(\frac{\Lieq}{\gamma_2^{\Lie}(\Lieq)}, K\right)=0.$ This completes the proof.
\end{proof}

\begin{Ex} \label{3.12}
The following is an example of Leibniz algebra satisfying the requirements of Corollary \ref{abelian der}.

Let $\Lieq$ be the three-dimensional Leibniz algebra with basis $\{a_1, a_2, a_3 \}$ and bracket operation given by $[a_2, a_2] = [a_3,a_3] = a_1$ and zero elsewhere (see algebra 2 (c) in \cite{CILL}). It is easy to check that
 $\gamma_2^{\Lie}(\Lieq) = Z_{\Lie}(\Lieq) = < \{ a_1 \} >$.
\end{Ex}

\section{$\Lie$-Central derivations and $\Lie$-centroids} \label{centroids}

\begin{De}
The {\Lie}-centroid of a Leibniz algebra $\Lieg$ is the set of all linear maps $d : \Lieg \to \Lieg$ satisfying  the identities
$$d([x,y])_{lie} = [d(x), y]_{lie} = [x, d(y)]_{lie}$$
for all $x, y \in \Lieg$. We denote this set by $\Gamma^{\Lie}(\Lieg)$.
\end{De}

\begin{Pro} \label{intersection}
For any Leibniz algebra $\Lieg$, $\Gamma^{\Lie}(\Lieg)$ is a subalgebra of ${\sf End}({\Lieg})$ such that $\Der_z^{\Lie}(\Lieg) = \Der^{\Lie}(\Lieg) \cap \Gamma^{\Lie}(\Lieg)$.
\end{Pro}
\begin{proof}
Assume that $d\in \mathsf{Der}^{\Lie}(\mathfrak{g})\cap \Gamma^{\mathsf{Lie}}(\mathfrak{g})$. For all $x,y\in \mathfrak{g}$, we have that $d([x,y]_{lie})=[d(x),y]_{lie}+[x,d(y)]_{lie}$; on the other hand,
$d([x,y]_{lie})=[x,d(y)]_{lie}$, hence $[d(x),y]_{lie}=0$ for any $y \in \Lieg$, that is $d(x) \in Z_{\Lie}(\Lieg)$

Conversely,  $\mathsf{Der}_{z}^{\mathsf{Lie}}(\mathfrak{g})$ is a subalgebra of $\mathsf{Der}^{\Lie}(\mathfrak{g})$  and for any  $d\in $ $\mathsf{Der}_{z}^{\mathsf{Lie}}(\mathfrak{g}),$ we have
$d([x,y]_{lie})=[d(x),y]_{lie}+[x,d(y)]_{lie}=0$, since $[d(x),y]_{lie}= 0$, $[x,d(y)]_{lie} = 0$, for any $x, y \in \Lieg$, hence $d \in \Gamma^{\Lie}(\Lieg)$.
\end{proof}

\begin{Pro}
Let $\Lieg$ be a Leibniz algebra. For any $d \in \Der^{\Lie}(\Lieg)$ and $\phi \in \Gamma^{\Lie}(\Lieg)$, the following statements hold:
\begin{enumerate}
\item[(a)] $\Der^{\Lie}(\Lieg) \subseteq N_{\Der^{\Lie}(\Lieg)}(\Gamma^{\Lie}(\Lieg))$.
\item[(b)] $d \circ \phi \in \Gamma^{\Lie}(\Lieg)$ if and only if $\phi \circ d \in \Der_z^{\Lie}(\Lieg)$.
\item[(c)] $d \circ \phi \in \Der^{\Lie}(\Lieg)$ if and only if $[d,\phi] \in \Der_z^{\Lie}(\Lieg)$.
\end{enumerate}
\end{Pro}
\begin{proof}
{\it (a)} Straightforward verification.

{\it (b)} Assume  $d\circ \phi \in \Gamma^{\Lie}({\Lieg})$. Then
\[
\begin{array}{rcl}
[\phi ,d]([x,y]_{lie}) &=& (\phi \circ d)([x,y]_{lie})- (d\circ \phi) ([x,y]_{lie})\\
&=& \left[ (\phi \circ d) \left( x\right) ,y\right] _{lie}+\left[ x,(\phi \circ d) \left( y\right) \right] _{lie}-[\left( d\circ \phi \right)(x),y]_{lie}  \\
&=&  [ [ \phi ,d](x),y]_{lie}+\left[ x,(\phi \circ d)\left( y\right) \right] _{lie}\\
&=& [\phi ,d]([x,y]_{{lie}})+\left[ x,(\phi \circ d) \left( y\right) \right] _{{lie}}
\end{array}
\]
Therefore $\left[ x,\left( \phi \circ d\right) \left( y\right) \right]_{{lie}}=0$. Similarly $[\left( d\circ \phi \right) (x),y]_{{lie}}=0$.

Conversely, assume  $\phi \circ d\in \mathsf{Der}_{z}^{\Lie}({\Lieg})$, then $[d,\phi ]([x,y]_{lie})=(d\circ \phi )([x,y]_{lie})-(\phi \circ d)([x,y]_{lie})$, hence
$(d\circ \phi )([x,y]_{lie})=[d,\phi ]([x,y]_{lie}),$ since $(\phi \circ d)([x,y]_{lie})=0$. Now it is a routine task to check that $[d,\phi ] \in \Gamma^{\Lie}(\Lieg)$, which completes the proof.

{\it (c)} Assume  $d\circ \phi \in \mathsf{Der}^{\Lie}({\Lieg}).$ A direct computation shows that  $\left[ \phi ,d\right] \in \Gamma^{\Lie}({\Lieg})$. On the other hand, it is easy to check that  $\left[ d, \phi \right] \in {\Der}^{\Lie}({\Lieg})$, therefore $\left[ \phi ,d\right] = - \left[ d, \phi \right] \in \Gamma^{\Lie}({\Lieg}) \cap {\Der}^{\Lie}({\Lieg})$. Proposition \ref{intersection} completes the proof.

Conversely, assume  $[d,\phi ]\in \mathsf{Der}_{z}^{\Lie}({\Lieg})$, then $(d\circ \phi )\left( [x,y]_{lie}\right) =[d,\phi ]\left( [x,y]_{lie}\right) +\left( \phi \circ d\right) \left( [x,y]_{lie} \right) =
\left( \phi \circ d\right) \left( [x,y]_{lie}\right)$. Now it is easy to check that $\phi \circ d$ is a $\Lie$-derivation of $\Lieg$.
\end{proof}

\begin{De}
Let $\Liem$ be a two-sided ideal of a Leibniz algebra $\Lieg$. Then $\Liem$ is said to be $\Gamma^{\Lie}(\Lieg)$-invariant if $\varphi(\Liem) \subset \Liem$ for all $\varphi \in \Gamma^{\Lie}(\Lieg)$.
\end{De}

\begin{Pro}
Let $\Lieg$ be a Leibniz algebra and $\Liem$ be a two-sided ideal of $\Lieg$. The following statements hold:
\begin{enumerate}
\item[(a)] $C_{\Lieg}^{\Lie}(\Liem, 0)$ is invariant under $\Gamma^{\Lie}(\Lieg)$.
\item[(b)] Every $\Lie$-perfect two-sided ideal $\Liem$ ($\Liem = \gamma_2^{\Lie}(\Liem)$) of $\Lieg$ is invariant under $\Gamma^{\Lie}(\Lieg)$.
\end{enumerate}
\end{Pro}
\begin{proof}
{\it (a)} Let $g\in C_{\Lieg}^{\Lie}({\Liem}, 0)$   and  $\varphi \in \Gamma^{\Lie}({\Lieg})$ be, then
$\varphi \left( g\right) \in C_{\mathfrak{g}}^{\Lie}({\Liem}, 0)$, since $[\varphi \left( g\right),m]_{lie}=\varphi \lbrack g,m]_{lie}=0,$ for all $m \in \Liem$

{\it (b)}  Let ${\Liem}$ be a ${\Lie}$-perfect two-sided ideal of ${\Lieg}$ and let $\varphi \in \Gamma^{\Lie}(\mathfrak{g})$ be, then any $x \in \Liem$ can be written as $x = \overset{n}{\underset{i=1}{\sum }}\lambda_{i}[m_{i1},m_{i2}]_{lie}, m_{i1},m_{i2}\in \Liem$, hence  $\varphi \left( x\right) = \overset{n}{\underset{i=1}{\sum}\lambda _{i}}
[\varphi \left( m_{i1}\right) ,m_{i2}]_{lie} \in {\Liem}$.
\end{proof}

\begin{Th} \label{isomorphism}
Let $\Liem$ be a nonzero $\Gamma^{\Lie}(\Lieg)$-invariant two-sided ideal of a Leibniz algebra $\Lieg$, $V(\Liem) = \{ \varphi \in \Gamma^{\Lie}(\Lieg) \mid \varphi(\Liem)=0 \}$  and $W = \Hom\left( \frac{\Lieg}{\Liem}, C_{\Lieg}^{\Lie}(\Liem, 0) \right)$ be. Define $$T(\Liem) = \left\{ f \in W \mid f[\overline{x}, \overline{y}]_{lie}=[f(\overline{x}), \overline{y}]_{lie} = [\overline{x}, f(\overline{y})]_{lie} \right\}$$ with $\overline{x} = x + \Liem$ and $\overline{y} = y +\Liem$. Then the following statements hold:
\begin{enumerate}
\item[(a)] $T(\Liem)$ is a vector subspace of $W$ isomorphic to $V(\Liem)$.
\item[(b)] If $\Gamma^{\Lie}(\Liem) = \mathbb{K}.\id_{\Liem}$, then $\Gamma^{\Lie}(\Lieg) = \mathbb{K}.\id_{\Lieg} \oplus V(\Liem)$ as vector spaces.
\end{enumerate}
\end{Th}
\begin{proof} {\it (a)} Define $\alpha :V({\Liem})\longrightarrow T({\Liem})$ by $\alpha \left( \varphi \right) \left( x+\mathfrak{m} \right) = \varphi \left( x\right)$.

Obviously,  $\alpha $ is an injective well-defined linear map and it is onto, since for every $f\in T(\mathfrak{m})$, set $\varphi _{f}:\mathfrak{g\longrightarrow g}$, $\varphi _{f}\left( x\right)
=f\left( x+\mathfrak{m}\right)$, for all $x\in \mathfrak{g}$. It is easy to check that
 $\varphi _{f}\in \Gamma ^{\mathsf{Lie}}(\mathfrak{g})$ and $\varphi_{f}\left( \mathfrak{m}\right) =0,$ so $\varphi _{f}\in V(\mathfrak{m})$. Moreover, $\alpha \left( \varphi _{f}\right)
\left( x+\mathfrak{m}\right) =$ $\varphi _{f}\left( x\right) =f\left( x+\mathfrak{m}\right)$.

{\it (b)} If $\Gamma ^{\mathsf{Lie}}(\mathfrak{m})=\mathbb{K}.\mathsf{Id}_{\mathfrak{m}}$, then for all $\psi \in \Gamma ^{\mathsf{Lie}}(\mathfrak{g})$, $\psi_{\mid _{\mathfrak{m}}}=\lambda. \mathsf{Id}_{\mathfrak{m}}$, for some $\lambda \in \mathbb{K}$.

If $\psi \neq \lambda. \mathsf{Id}_{\mathfrak{g}}$, define $\varphi : \Lieg \to \Lieg$ by $\varphi \left( x\right) =\lambda x$, then $\varphi \in \Gamma ^{\mathsf{Lie}}(\mathfrak{g})$
and $\psi -\varphi \in V(\mathfrak{m})$.
Clearly, $\psi =\varphi +\left( \psi -\varphi \right) \in \mathbb{K}.\id_{\Lieg} + V(\Liem)$. Furthermore it is evident that $\mathbb{K}.\mathsf{Id}_{\Lieg}\cap V(\mathfrak{m})=0$, which completes the proof.
\end{proof}

\begin{Co}
If $\mathbb{K}$ is a field of zero characteristic, then the following equalities hold:
$$\Der_z^{\Lie}(\Lieg) = V(\gamma_2^{\Lie}(\Lieg)) =  T(\gamma_2^{\Lie}(\Lieg))$$
\end{Co}
\begin{proof}
If $d\in {\Der}_{z}^{\Lie}(\mathfrak{g})$,  then  $d\in {\Der}^{\Lie}({\Lieg})\cap \Gamma^{\Lie}({\Lieg})$ by Proposition \ref{intersection}, hence $[d\left( x\right)
,y]_{lie}=[x,d(y)]_{lie}=0$, so $d\in V(\gamma _{2}^{\mathsf{Lie}}(\mathfrak{g}))$.

Conversely, if $d\in V(\gamma _{2}^{\Lie}({\Lieg}))$, then $d\in \Gamma ^{\mathsf{Lie}}(\mathfrak{g})$ and $d([x,y]_{lie})=0$, so $d([x,y]_{lie})$ $=[d\left( x\right) ,y]_{lie}=[x,d\left( y\right) ]_{lie}=0.$
Hence $d([x,y]_{lie})=[d\left( x\right) ,y]_{lie}+[x,d\left( y\right) ]_{lie}=0$, which implies that $d\in \mathsf{Der}_{z}^{\mathsf{Lie}}(\mathfrak{g}).$

The second equality is provided by Theorem \ref{isomorphism} since $\gamma_2^{\Lie}(\Lieg)$ is $\Gamma^{\Lie}(\Lieg)$-invariant.
\end{proof}

\begin{Th}
Let $\Lieg$ be a Leibniz algebra such that $\Lieg = \Lieg_1 \oplus \Lieg_2$, where $\Lieg_1, \Lieg_2$ are two-sided ideals of $\Lieg$. Then  the following isomorphism of $\mathbb{K}$-vector spaces holds:
$$\Gamma^{\Lie}(\Lieg) \cong \Gamma^{\Lie}(\Lieg_1) \oplus \Gamma^{\Lie}(\Lieg_2) \oplus C_1 \oplus C_2$$
where $C_i = \{ \varphi \in \Hom(\Lieg_i, \Lieg_j) \mid \varphi(\Lieg_i) \subseteq Z_{\Lie}(\Lieg_j) ~ \text{and}~ \varphi(\gamma_2^{\Lie}(\Lieg_i))=0 ~ \text{for}~
 1 \leq i \neq j \leq 2 \}$.
 \end{Th}
\begin{proof}
Let $\pi _{i}:\mathfrak{g}\longrightarrow \mathfrak{g}_{i}$ be the canonical projection for $i=1,2$. Then $\pi _{1},\pi _{2}\in \Gamma^{\Lie}(\mathfrak{g)}$
and $\pi_{1}+\pi _{2}=\id_{\mathfrak{g}}$.

So we have for $\varphi \in \Gamma^{\Lie}(\mathfrak{g)}$ that $\varphi =\pi _{1} \circ \varphi \circ \pi _{1}+\pi _{1} \circ \varphi \circ \pi _{2}+\pi _{2} \circ \varphi \circ \pi _{1}+\pi _{2} \circ \varphi \circ \pi _{2}$

Note that $\pi _{i} \circ \varphi \circ \pi _{j}\in \Gamma^{\mathsf{Lie}}(\mathfrak{g)}$ for $i,j=1, 2$. So, by the above equality it follows that
$$\Gamma^{\Lie}(\mathfrak{g)} = \pi _{1}\Gamma^{\Lie}(\mathfrak{g)}\pi _{1}\oplus \pi _{1}\Gamma^{\Lie}(\mathfrak{g)}\pi_{2}\oplus \pi _{2}\Gamma^{\Lie}(\mathfrak{g)}\pi _{1}\oplus \pi_{2}\Gamma^{\Lie}(\mathfrak{g)}\pi _{2}$$
 as vector spaces. Indeed, it is enough to show that $\pi _{i}\Gamma^{\Lie}(\mathfrak{g)} \pi_{k}\cap \pi_{l}\Gamma^{\Lie}(\mathfrak{g)}\pi _{j}=0$ for $i, j, k, l=1, 2$, such that $(i, j) \neq (k, l)$. For instance
$\pi _{2}\Gamma^{\Lie}(\mathfrak{g)}\pi _{1}\cap \pi _{1}\Gamma^{\Lie}(\mathfrak{g)}\pi_{2}=0$, since for any $\beta \in \pi _{2}\Gamma^{\Lie}(\mathfrak{g)}\pi _{1}\cap
\pi _{1}\Gamma^{\Lie}(\mathfrak{g)}\pi _{2}$, there are some $f_{1},f_{2}\in \Gamma^{\Lie}(\mathfrak{g)}$ such that $\beta = \pi_{2} \circ f_{1} \circ \pi _{1}=\pi _{1} \circ f_{2} \circ \pi _{2}$, then
$\beta \left( x\right) =\pi _{1} \circ f_{2} \circ \pi _{2}\left( x\right) =\pi_{1} \circ f_{2} \circ \pi _{2}\left( \pi _{2}\left( x\right) \right) =\pi _{2} \circ f_{1} \circ \pi_{1}\left( \pi _{2}\left( x\right) \right) =
\pi _{2} \circ f_{1}\left( 0\right) =0$, for all $x\in \mathfrak{g}$. Hence $\beta =0$. Other cases can be checked in a similar way.

Now let's denote $\Gamma^{\Lie}(\mathfrak{g)}_{ij}=\pi _{i}\Gamma ^{\mathsf{Lie}}(\mathfrak{g)}\pi _{j}$, $i, j = 1, 2$. We claim that  the following isomorphisms of vector spaces hold:
$$\Gamma^{\Lie}(\mathfrak{g)}_{11}\cong \Gamma^{\Lie}(\mathfrak{g}_{1}), ~ \Gamma^{\Lie}(\mathfrak{g)}_{22}\cong \Gamma^{\Lie}(\mathfrak{g}_{2}),~ \Gamma^{\Lie}(\mathfrak{g)}_{12}\cong C_{2}, ~
 \Gamma^{\Lie}(\mathfrak{g)}_{21}\cong C_{1}$$

For $\varphi \in \Gamma^{\Lie}(\mathfrak{g)}_{11}$, $\varphi \left(\mathfrak{g}_{2}\right) =0$  so $\varphi_{\mid \mathfrak{g}_{1}}\in \Gamma^{\Lie}(\mathfrak{g}_{1}\mathfrak{)}$. Now, considering
 $\Gamma^{\Lie}(\mathfrak{g}_{1}\mathfrak{)}$ as a subalgebra of $\Gamma^{\Lie}(\mathfrak{g)}$ such that for any $\varphi _{0}\in \Gamma^{\Lie}(\mathfrak{g}_{1})$, $\varphi_0$
vanishes on $\mathfrak{g}_{2}$, that is, $\varphi _{0}\left( x_{1}\right) =$ $\varphi _{0}\left( x_{2}\right)$, \ $\varphi _{0}\left( x_{2}\right) =$ $0$, for all $x_{1}\in \mathfrak{g}_{1}$
and $x_{2}\in \mathfrak{g}_{2}$. Then $\varphi _{0}\in \Gamma^{\Lie}(\mathfrak{g)}$ and $\varphi_{0}\in \Gamma^{\Lie}(\mathfrak{g)}_{11}$. Therefore, $\Gamma^{\Lie}(\mathfrak{g)}_{11}\cong
\Gamma^{\Lie}(\mathfrak{g}_{1})$ by means of the isomorphism $\sigma :\Gamma^{\Lie}(\mathfrak{g)}_{11}\longrightarrow \Gamma^{\Lie}(\mathfrak{g}_{1}\mathfrak{)}$, $\sigma \left( \varphi \right)
=\varphi_{\mid  \mathfrak{g}_{1}}$, for all $\varphi \in \Gamma ^{\mathsf{Lie}%
}(\mathfrak{g)}_{11}$.

The isomorphism $\Gamma ^{\mathsf{Lie}}(\mathfrak{g)}_{22}\cong \Gamma^{\mathsf{Lie}}(\mathfrak{g}_{2}\mathfrak{)}$ can be proved in an analogous way.

 $\Gamma^{\Lie}(\mathfrak{g)}_{12}\cong C_{2}$. Indeed, for any  $\varphi \in \Gamma^{\Lie}(\mathfrak{g)}_{12}$ there exists a $\varphi _{0}\in \Gamma^{\Lie}(\mathfrak{g)}$ such that
  $\varphi=\pi _{1} \circ \varphi _{0} \circ \pi _{2}$. For $x_{k}=(x_{k}^{1}, x_{k}^{2}) \in \mathfrak{g}$, where $x_{k}^{i}\in
\mathfrak{g}_{i}$, $i=1,2$, $k=1,2$, we have
\[
\begin{array}{rcl}
\varphi \left( [ x_{1},x_{2}]_{lie}\right) &=& \pi _{1} \circ \varphi_{0} \circ \pi_{2}\left( [x_{1},x_{2}]_{lie}\right) =\pi _{1} \circ \varphi _{0} \circ \pi _{2}\left(
[(x_{1}^{1}, x_{1}^{2}), (x_{2}^{1}, x_{2}^{2})]_{lie}\right) \\
&= & \pi _{1}\varphi _{0}\left( [x_{1}^{2},x_{2}^{2}]_{lie}\right) =\pi _{1}\left([\varphi _{0}\left( x_{1}^{2}\right) ,x_{2}^{2}]_{lie}\right) =0
\end{array}
\]
hence $\varphi (\gamma _{2}^{\mathsf{Lie}}(\mathfrak{g}))=0$. On the other hand, $[\varphi \left( x_{1}\right) ,x_{2}]_{lie}=\varphi \left( \lbrack x_{1},x_{2}]_{lie}\right) =0$,
so, $\varphi (\mathfrak{g})\subseteq Z_{\mathsf{Lie}}(\mathfrak{g})$ and $\varphi (\gamma _{2}^{\mathsf{Lie}}(\mathfrak{g}))=0$.

It follows that $\varphi_{\mid \mathfrak{g}_{2}}\left( \mathfrak{g}_{2}\right) \subseteq Z_{\Lie}(\mathfrak{g}_{1})$ and $\varphi_{\mid \mathfrak{g}_{2}}(\gamma _{2}^{\Lie}(\mathfrak{g}_{2}))=0$,
hence $\varphi_{\mid {\mathfrak{g}_{2}}}\in C_{2}.$

Conversely, for $\varphi \in C_{2}$, expanding $\varphi $ on $\mathfrak{g}$ by $\varphi \left( \mathfrak{g}_{1}\right) =0$, we have $\pi _{1} \circ \varphi \circ \pi_{2}=\varphi$
and so $\varphi \in $ $\Gamma^{\Lie}(\mathfrak{g)}_{12}$. Hence $\Gamma^{\Lie}(\mathfrak{g)}_{12}\cong C_{2}$, by means the
 isomorphism $\tau : \Gamma^{\Lie}(\mathfrak{g)}_{12}\longrightarrow C_{2}$, $\tau \left( \varphi \right) =\varphi_{\mid \mathfrak{g}_{2}}$ \ for all $\varphi \in \Gamma ^{\mathsf{Lie}}(\mathfrak{g)}_{12}$.

 Similarly, can be  proved that $\Gamma ^{\mathsf{Lie}}(\mathfrak{g)}_{21}\cong C_{1}$, which completes the proof.
\end{proof}


\section{${\sf ID}^{\Lie}$-derivations}\label{almost}

\begin{De}
A $\Lie$-derivation $d : \Lieg \to \Lieg$ is said to be an {\sf ID}-$\Lie$-derivation if $d(\Lieg) \subseteq \gamma_2^{\Lie}(\Lieg)$. The set of all {\sf ID}-$\Lie$-derivations of $\Lieg$ is denoted by ${\sf ID}^{\Lie}(\Lieg)$.

An {\sf ID}-$\Lie$-derivation $d : \Lieg \to \Lieg$ is said to be ${\sf ID}_*$-${\Lie}$-derivation if $d$ vanishes on the $\Lie$-central elements of $\Lieg$. The set of all ${\sf ID}_*$-${\Lie}$-derivations of $\Lieg$ is denoted by ${\sf ID}_*^{\Lie}(\Lieg)$.
\end{De}

It is obvious that ${\sf ID}^{\Lie}(\Lieg)$ and ${\sf ID}_*^{\Lie}(\Lieg)$ are subalgebras of $\Der^{\Lie}(\Lieg)$ and
\begin{equation} \label{inclusion}
{\Der}_c^{\Lie}(\Lieg) \subseteq {\sf ID}_*^{\Lie}(\Lieg) \subseteq {\sf ID}^{\Lie}(\Lieg)
\end{equation}
where ${\Der}_c^{\Lie}(\Lieg)$ is the subspace of ${\Der}^{\Lie}(\Lieg)$ given by $\{ d \in \Der^{\Lie}(\Lieg) \mid d(x) \in [x, \Lieg]_{lie}, \forall x \in \Lieg \}$.  These kinds of derivations are called almost inner {\Lie}-derivations of $\Lieg$.

\begin{Ex}
Let $\Lieg$ be the three-dimensional Leibniz algebra with basis $\{a_1, a_2, a_3 \}$ and bracket operation given by $[a_2, a_2] = [a_3, a_3] = a_1$ and zero elsewhere (algebra 2 (c) in \cite{CILL}). The right multiplications {\Lie}-derivations $R_x, x \in \Lieg$, are examples of almost inner {\Lie}-derivations.
\end{Ex}

\begin{De}
An almost inner {\Lie}-derivation $d$ is said to be central almost inner {\Lie}-derivation if there exists an $x \in Z^l(\Lieg)$ such that $(d-R_x) (\Lieg) \subseteq Z_{\Lie}(\Lieg)$.

We denote the $\mathbb{K}$-vector space of all central almost inner {\Lie}-derivation by ${\Der}_{cz}^{\Lie}(\Lieg)$.
\end{De}

\begin{Th} \label{IDstar}
Let $\Lieg$ and $\Lieq$ be two $\Lie$-isoclinic Leibniz algebras. Then ${\sf ID}_*^{\Lie}(\Lieg) \cong {\sf ID}_*^{\Lie}(\Lieq)$.
\end{Th}
\begin{proof}
 Let  $(\eta,\xi)$ be the \Lie-isoclinism between $\Lieg$ and $\Lieq$  and let $\alpha\in {\sf ID}_*^{\Lie}(\Lieg).$ Consider the map $\zeta_{\alpha}:\Lieq\to\Lieq$ defined by $\zeta_{\alpha}(y):=\xi(\alpha(x)),$ where $y+ Z_{\Lie}(\Lieq)=\eta(x+ Z_{\Lie}(\Lieg)).$ Clearly $\zeta_{\alpha}$ is a well-defined linear map since $\alpha$ and $\xi$ are linear maps, and  if $y\in Z_{\Lie}(\Lieq),$ then $x\in  Z_{\Lie}(\Lieg)$ and thus $\zeta_{\alpha}(y)=\xi(\alpha(x))=\zeta(0)=0.$ To show that $\zeta_{\alpha}$ is a \Lie-derivation, let $y_1,y_2\in\Lieq$ and $x_1,x_2\in\Lieg$ such that   $y_i+ Z_{\Lie}(\Lieq)=\eta(x_i+ Z_{\Lie}(\Lieg)),~i=1,2.$ Then
$$
\begin{aligned}
\zeta_{\alpha}([y_1,y_2]_{lie})&=\xi(\alpha([x_1,x_2]_{lie}))\\&=\xi([\alpha(x_1),x_2]_{lie})+\xi([x_1,\alpha(x_2)]_{lie})~~\mbox{by}~~\cite[Prop. ~ 3.8]{BC}\\
&=[\xi(\alpha(x_1)),y_2]_{lie}+[y_1,\xi(\alpha(x_2))]_{lie}\\&=[\zeta_{\alpha}(y_1),y_2]_{lie}+[y_1,\zeta_{\alpha}(y_2)]_{lie}.
\end{aligned}$$
Moreover, since $\alpha(\Lieg)\subseteq \gamma_2^{\Lie}(\Lieg)$ and $\xi$ is an isomorphism, it follows that  $\zeta_{\alpha}(\Lieq)\subseteq \gamma_2^{\Lie}(\Lieq).$ Therefore $\zeta_{\alpha}\in{\sf ID}_*^{\Lie}(\Lieq).$  Now consider the map $\zeta: {\sf ID}_*^{\Lie}(\Lieg)\to {\sf ID}_*^{\Lie}(\Lieq)$ defined by $\zeta(\alpha)=\zeta_{\alpha}.$ We claim that $\xi$ is a \Lie-homomorphism. Indeed, for $\alpha_1,\alpha_2\in{\sf ID}_*^{\Lie}(\Lieg) ,$ we have for all $y\in\Lieq,$ and $x\in\Lieg$ such that $y+ Z_{\Lie}(\Lieq)=\eta(x+ Z_{\Lie}(\Lieg)),$
$$
\begin{aligned}
\zeta([\alpha_1,\alpha_2])(y)&=\zeta_{[\alpha_1,\alpha_2]}(y)=\xi([\alpha_1,\alpha_2](x))=\xi(\alpha_1(\alpha_2(x))-\alpha_2(\alpha_1(x)))\\&=\xi(\alpha_1(\alpha_2(x)))-\xi(\alpha_2(\alpha_1(x)))\\&=\zeta_{\alpha_1}
(\xi(\alpha_2(x))-\zeta_{\alpha_2}(\xi(\alpha_1(x))~~\mbox{by}~~\cite[Prop. ~ 3.8]{BC}\\&=\zeta_{\alpha_1}(\zeta_{\alpha_2}(y))-\zeta_{\alpha_2}(\zeta_{\alpha_1}(y))=[\zeta_{\alpha_1},\zeta_{\alpha_2}](y)=[\zeta(\alpha_1),\zeta(\alpha_2)](y).
\end{aligned}$$
Hence $\zeta([\alpha_1,\alpha_2])=[\zeta(\alpha_1),\zeta(\alpha_2)].$ Conversely, let $\beta\in {\sf ID}_*^{\Lie}(\Lieq). $ By using the inverse \Lie-isoclinism  $(\eta^{-1},\xi^{-1}),$ we similarly construct a homomorphism $\zeta': {\sf ID}_*^{\Lie}(\Lieq)\to {\sf ID}_*^{\Lie}(\Lieg)$ defined by $\zeta'(\beta)=\zeta'_{\beta}$ where $\zeta'_{\beta}(x)=\xi^{-1}(\beta(y))$ with $y+ Z_{\Lie}(\Lieq)=\eta(x+ Z_{\Lie}(\Lieg)).$ It is clear that $(\zeta'\circ\zeta)(\alpha)(x)=\zeta'(\zeta(\alpha))(x)=\zeta'_{\zeta(\alpha)}(x)=\xi^{-1}(\zeta(\alpha)(y))=\xi^{-1}(\zeta_{\alpha}(y))=\xi^{-1}(\xi(\alpha(x)))=\alpha(x).$ So $\zeta'\circ\zeta=\id_{{\sf ID}_*^{\Lie}(\Lieg)}.$ Similarly, one shows that $\zeta\circ\zeta'=\id_{{\sf ID}_*^{\Lie}(\Lieq)}.$ Therefore ${\sf ID}_*^{\Lie}(\Lieg) \cong {\sf ID}_*^{\Lie}(\Lieq)$.
\end{proof}

\begin{Co}
Let $\Lieg$ and $\Lieq$ be two $\Lie$-isoclinic Leibniz algebras. Then ${\Der}_c^{\Lie}(\Lieg) \cong {\Der}_c^{\Lie}(\Lieq)$.
\end{Co}
\begin{proof}
 Let  $(\eta,\xi)$ be the \Lie-isoclinism between $\Lieg$ and $\Lieq$  and let $\alpha\in {\Der}_c^{\Lie}(\Lieg).$ Consider again the map  $\zeta_{\alpha}:\Lieq\to\Lieq$ defined by $\zeta_{\alpha}(y):=\xi(\alpha(x)),$ where $y+ Z_{\Lie}(\Lieq)=\eta(x+ Z_{\Lie}(\Lieg)),$ given in the proof of Theorem \ref{IDstar}.  Since $\alpha(x)\in[x,\Lieg]_{lie}$ and $\xi$ is an isomorphism, it is clear that $\zeta_{\alpha}(y)\in[y,\Lieq]_{lie}$   for all  $y\in\Lieq.$
So $\zeta_{\alpha}\in {\Der}_c^{\Lie}(\Lieq).$  So the restriction $\zeta_{\mid {\Der}_c^{\Lie}(\Lieg) }: {\Der}_c^{\Lie}(\Lieg) \to {\Der}_c^{\Lie}(\Lieq)$ of the map $\zeta$ in the proof of Theorem \ref{IDstar}  to ${\Der}_c^{\Lie}(\Lieg)$ is a homomorphism. Similarly, by using the inverse \Lie-isoclinism  $(\eta^{-1},\xi^{-1}),$ one obtains a homomorphism by taking  the restriction $\zeta'_{\mid {\Der}_c^{\Lie}(\Lieq) }: {\Der}_c^{\Lie}(\Lieq) \to {\Der}_c^{\Lie}(\Lieg)$ of the map $\zeta'$ in the proof of Theorem \ref{IDstar}  to ${\Der}_c^{\Lie}(\Lieq).$ It is clear that   $\zeta\circ\zeta'_{\mid {\Der}_c^{\Lie}(\Lieq) }=\id_{{\Der}_c^{\Lie}(\Lieq) }$ and  $\zeta'\circ\zeta_{\mid {\Der}_c^{\Lie}(\Lieg) }=\id_{{\Der}_c^{\Lie}(\Lieg) }.$  Therefore ${\Der}_c^{\Lie}(\Lieg) \cong {\Der}_c^{\Lie}(\Lieq)$.
\end{proof}
\bigskip

For any $d \in {\Der}_z^{\Lie}(\Lieg)$, the map $\psi_d : \frac{\Lieg}{\gamma_2^{\Lie}(\Lieg)} \to Z_{\Lie}(\Lieg)$ given by $\psi_d(g+\gamma_2^{\Lie}(\Lieg)) = d(g)$ is a linear map. It is easy to show that the linear map $\psi : {\Der}_z^{\Lie}(\Lieg) \to T\left(\frac{\Lieg}{\gamma_2^{\Lie}(\Lieg)}, Z_{\Lie}(\Lieg)\right)$, $\psi(d) = \psi_d$ is bijective. Therefore, dim$ \left( {\Der}_z^{\Lie}(\Lieg) \right)$ = dim $\left( T\left(\frac{\Lieg}{\gamma_2^{\Lie}(\Lieg)}, Z_{\Lie}(\Lieg)\right) \right)$ for any finite-dimensional Leibniz algebra $\Lieg$.

\begin{Co} \label{5.4}
Let $\Lieg$ be a finite-dimensional Leibniz algebra such that  $[\Lieg, \Lieg] = \gamma_2^{\Lie}(\Lieg)$ and  $Z_{\Lie}(\Lieg) \subseteq  Z^r(\Lieg)$. Then ${\sf ID}_*^{\Lie}(\Lieg) = {\Der}_z^{\Lie}(\Lieg)$ if and only if $\gamma_2^{\Lie}(\Lieg) = Z_{\Lie}(\Lieg)$.
\end{Co}
\begin{proof}
 Assume that $\gamma_2^{\Lie}(\Lieg) = Z_{\Lie}(\Lieg)$. It is clear that for all $d\in{\Der}_z^{\Lie}(\Lieg),$ $d(\Lieg)\subseteq Z_{\Lie}(\Lieg) \iff d(\Lieg)\subseteq \gamma_2^{\Lie}(\Lieg)$ and
$d(Z_{\Lie}(\Lieg))=d(\gamma_2^{\Lie}(\Lieg))=0.$ Therefore ${\sf ID}_*^{\Lie}(\Lieg) = {\Der}_z^{\Lie}(\Lieg).$

Conversely,  assume that ${\sf ID}_*^{\Lie}(\Lieg) = {\Der}_z^{\Lie}(\Lieg).$ Then since  $[\Lieg, \Lieg] = \gamma_2^{\Lie}(\Lieg)$ and $Z_{\Lie}(\Lieg) \subseteq  Z^r(\Lieg)$,  it follows that  the map $R_x  : \Lieg \to \Lieg, R_x(y)=[y, x],$ is a {\Lie}-derivation; moreover  it is easy to check that $R_x \in {\sf ID}_*^{\Lie}(\Lieg) = {\Der}_z^{\Lie}(\Lieg)$,  hence $R_{x}(y)\in Z_{\Lie}(\Lieg),$ for all $y\in\Lieg.$ Therefore $ {\ze}_{2}^{\Lie}({\Lieg}) =\Lieg,$ and thus $\Lieg$ is \Lie-nilpotent of class 2 by Theorem \ref{2.4}. Now, by   \cite[Corollary 4.1]{BC2}, there is a $\Lie$-stem Leibniz algebra $\Lieq$ \Lie-isoclinic to $\Lieg.$  Denote this \Lie-isoclinism by $(\eta,\xi).$ Since $\Lieg$ is $\Lie$-nilpotent Leibniz algebra of class 2, so is $\Lieq,$ thanks to Lemma \ref{gNilq}. This implies that $ [\Lieg,\Lieg]_{\Lie}\stackrel{\xi}\cong  [\Lieq,\Lieq]_{\Lie}=Z_{\Lie}(\Lieq),$ and $\frac{\Lieg}{Z_{\Lie}(\Lieg)}\stackrel{\eta}\cong\frac{\Lieq}{Z_{\Lie}(\Lieq)}\cong \frac{ \Lieq}{[\Lieq,\Lieq]_{\Lie}}.$ It follows by  Theorem \ref{IDstar}, the first implication  and Lemma \ref{DT} that
$$
\begin{aligned}
{\rm dim}({\Der}_z^{\Lie}(\Lieg))&={\rm dim}({\sf ID}_*^{\Lie}(\Lieg))={\rm dim}({\sf ID}_*^{\Lie}(\Lieq))={\rm dim}({\Der}_z^{\Lie}(\Lieq))\\&={\rm dim}\left(T\left(\frac{\Lieq}{[\Lieq, \Lieq]_{\Lie} }, Z_{\Lie}(\Lieq)\right)\right)={\rm dim}\left(T\left(\frac{\Lieq}{Z_{\Lie}(\Lieq)}, [\Lieq, \Lieq]_{\Lie}\right)\right)\\&= {\rm dim}\left(T\left(\frac{\Lieg}{Z_{\Lie}(\Lieg)}, [\Lieg, \Lieg]_{\Lie}\right)\right)= {\rm dim}\left(Z({\Der}_z^{\Lie}(\Lieg)) \right).
\end{aligned}$$ The latter equality is due to Theorem \ref{3.10} since $\Lieg$ is \Lie-nilpotent of class 2. Therefore ${\Der}_z^{\Lie}(\Lieg)$ is abelian. We now conclude by Corollary \ref{abelian der} that $\gamma_2^{\Lie}(\Lieg)= Z_{\Lie}(\Lieg)$.
\end{proof}

\begin{Rem}
Let us observe that the requirements $[\Lieg, \Lieg] = \gamma_2^{\Lie}(\Lieg)$ and  $Z_{\Lie}(\Lieg) \subseteq  Z^r(\Lieg)$ in Corollary \ref{5.4} are not needed in the absolute case, but in our relative setting they are absolutely necessary as the following counterexample shows:
let $\Lieg$ be the four-dimensional complex Leibniz algebra with basis $\{ a_1, a_2, a_3, a_4 \}$ and bracket operation given by $[a_1, a_2] = - [a_2, a_1] = a_4; [a_3,a_3] = a_4$ and zero elsewhere (class $\frak{R}_{21}$ in \cite[Theorem 3.2]{AOR}). It is easy to check that $[\Lieg, \Lieg] = \langle \{ a_4 \} \rangle = \gamma_2^{\Lie}(\Lieg)$, $Z_{\Lie}(\Lieg) = \langle \{ a_1, a_2, a_4 \} \rangle$ and $Z^{r}(\Lieg) = \langle \{ a_4 \} \rangle$.

Consider the {\Lie}-derivation $R_{a_1}$, which belongs to  ${\sf Der}_z^{\Lie}(\Lieg)$. However $R_{a_1} \notin {\sf ID}_*^{\Lie}(\Lieg)$ since $R_{a_1}$ doesn't vanish on $Z_{\Lie}(\Lieg)$.
\end{Rem}

\begin{Ex}
The three-dimensional complex Leibniz algebra with basis $\{ a_1, a_2, a_3 \}$ and bracket operation given by $[a_2, a_2] = \gamma a_1, \gamma \in \mathbb{C}; [a_3,a_2] =  [a_3, a_3] = a_1$ and zero elsewhere (class 2 (a) in \cite{CILL}) satisfies the requirements of Corollary \ref{5.4}, since $[\Lieg, \Lieg]  = \gamma_2^{\Lie}(\Lieg) = Z_{\Lie}(\Lieg) = Z^r(\Lieg)  = \langle \{ a_1 \} \rangle$.

\end{Ex}

\begin{Th} \label{inequality}
Let $\Lieg$ be a Leibniz algebra such that $\gamma_2^{\Lie}(\Lieg)$ is finite dimensional and $\frac{\Lieg}{Z_{\Lie}(\Lieg)}$ is generated by $p$ elements. Then
$${\rm dim}({\sf ID}_*^{\Lie}(\Lieg)) \leq p \cdot {\rm dim}(\gamma_2^{\Lie}(\Lieg))$$
\end{Th}
\begin{proof}
 Consider the map  $\alpha: {\sf ID}_*^{\Lie}(\Lieg) \to T\left(\frac{\Lieg}{Z_{\Lie}(\Lieg)}, \gamma_2^{\Lie}(\Lieg)\right)$ defined by $d \mapsto d^*$ such that $d^*(x+Z_{\Lie}(\Lieg))=d(x).$  Then $\alpha$ is a well-defined injective linear map. It follows that ${\rm dim}({\sf ID}_*^{\Lie}(\Lieg)) \leq  {\rm dim}\left(T\left(\frac{\Lieg}{Z_{\Lie}(\Lieg)}, \gamma_2^{\Lie}(\Lieg)\right)\right)=p \cdot {\rm dim}(\gamma_2^{\Lie}(\Lieg))$
\end{proof}

\begin{Ex}
 Now we present two examples illustrating the inequality in Theorem \ref{inequality}.
\begin{enumerate}
\item[(a)] Let $\Lieg$ be the three-dimensional Leibniz algebra with basis $\{a_1, a_2, a_3 \}$ and bracket operation given by $[a_2, a_3] = - [a_3, a_2] = a_2,  [a_3, a_3]=a_1$ and zero elsewhere (class 2 (f) in \cite{CILL}).

It is an easy task to check that $\frac{\Lieg}{Z_{\Lie}(\Lieg)} = \langle \{  \overline{a}_3 \} \rangle$, hence the number of generators is $p=1$. Moreover $\gamma_2^{\Lie}(\Lieg) = \langle \{ a_1 \}\rangle$.  Also it can be checked that an element $d \in {\sf ID}_*^{\Lie}(\Lieg)$  is represented by a matrix of the form $\left( \begin{array}{ccc} 0 & 0 & a_{13} \\ 0 & 0 & 0 \\ 0 &  0 & 0 \end{array} \right)$.  Hence  ${\rm dim}({\sf ID}_*^{\Lie}(\Lieg)) = 1 \leq 1 \cdot 1$.

\item[(b)] Let $\Lieg$ be the four-dimensional Leibniz algebra with basis $\{a_1, a_2, a_3, a_4 \}$ and bracket operation given by $[a_1, a_4]=a_1, [a_2, a_4]=a_2$ and zero elsewhere (class ${\cal R}_2$ in \cite[Theorem 2.7]{CK}).

It is an easy task to check that $\frac{\Lieg}{Z_{\Lie}(\Lieg)} = \langle \{ \overline{a}_1, \overline{a}_2,\overline{a}_4 \} \rangle$, hence the number of generators is $p=3$. Moreover $\gamma_2^{\Lie}(\Lieg) = \langle \{ a_1, a_2 \}\rangle$.  Also it can be checked that an element $d \in {\sf ID}_*^{\Lie}(\Lieg)$  is represented by a matrix of the form $\left( \begin{array}{cccc} a_{11} & a_{12}  & 0 & 0 \\ a_{21} & a_{22} & 0 &0  \\ 0 &  0 &0 & 0 \\ 0 &  0 &0 & 0  \end{array} \right)$.  Hence  ${\rm dim}({\sf ID}_*^{\Lie}(\Lieg)) = 4 \leq 3 \cdot 2$.
\end{enumerate}
\end{Ex}

\begin{Co} \label{bounds}
Let $\Lieg$ be a Leibniz algebra such that  $Z^r(\Lieg)  =  Z_{\Lie}(\Lieg)$, $[\Lieg, \Lieg] = \gamma_2^{\Lie}(\Lieg)$ is finite dimensional and   $\frac{\Lieg}{Z_{\Lie}(\Lieg)}$ is generated by $p$ elements. Then
$${\rm dim} \left( \frac{\Lieg}{Z_{\Lie}(\Lieg)} \right) \leq p \cdot {\rm dim}(\gamma_2^{\Lie}(\Lieg))$$
\end{Co}
\begin{proof}
 Under these hypothesis, we have from the proof of Corollary \ref{5.4} that $R_x \in {\sf ID}_*^{\Lie}(\Lieg)$ for all $x \in \Lieg$. Now, consider the $\mathbb{K}$-linear map  $\beta :\frac{\Lieg}{Z_{\Lie}(\Lieg)}  \to {\sf ID}_*^{\Lie}(\Lieg)$ defined by $x+Z_{\Lie}(\Lieg) \mapsto R_x$, which is an injective well-defined linear map, since  $\Ker(\beta) = \frac{Z^r(\Lieg)}{Z_{\Lie}(\Lieg)} = 0$. Hence
${\rm dim} \left( \frac{\Lieg}{Z_{\Lie}(\Lieg)} \right) \leq {\rm dim}\left( {\sf ID}_*^{\Lie}(\Lieg) \right)$. Now Theorem \ref{inequality} completes the proof.
\end{proof}

\begin{Ex} \label{5.8}
The three-dimensional non-Lie Leibniz algebra with basis $\{a_1, a_2, a_3 \}$ and bracket operation $[a_3, a_3] = a_1$ and zero elsewhere, satisfies the requirements of Corollary \ref{bounds}.
\end{Ex}

\begin{De}
A Leibniz algebra $\Lieg$ of dimension $n$ is said to be $\Lie$-filiform  (or $1$-$\Lie$-filiform) if ${\rm dim}(\gamma_i^{\Lie}(\Lieg)) = n-i,~ 2 \leq i \leq n$.

\end{De}

$\Lie$-filiform Leibniz algebras are $\Lie$-nilpotent of class $n-1$.

\begin{Co} \label{5.10}
Let $\Lieg$ be an $n$-dimensional Leibniz algebra such that $Z^r(\Lieg) = Z_{\Lie}(\Lieg) \subseteq  Z^l(\Lieg)$ and it attains the upper bound of Corollary \ref{bounds}. If $\Lieg$ is $\Lie$-filiform, then $n=3$.
\end{Co}
\begin{proof}
 If $\Lieg$ is $\Lie$-filiform, then ${\rm dim}(\gamma_2^{\Lie}(\Lieg)) = n-2,~n\geq 2.$  By the assumption on Corollary \ref{bounds},  $p={\rm dim} \left( \frac{\Lieg}{Z_{\Lie}(\Lieg)} \right) = p \cdot {\rm dim}(\gamma_2^{\Lie}(\Lieg))=p(n-2),$ which implies that $n=3$.
\end{proof}

\begin{Rem}
Example \ref{5.8} provides a {\Lie}-filiform Leibniz algebra which illustrates Corollary \ref{5.10}.
\end{Rem}

\begin{Pro}
Let $\Lieg$ be a Leibniz algebra. Then the following statements hold:
\begin{enumerate}
\item[(a)] Let $d \in {\sf Der}_c^{\Lie}(\Lieg)$ be. Then $d(\Lieg) \subseteq \gamma_2^{\Lie}(\Lieg),$ $d(Z_{\Lie}(\Lieg)) = 0$ and $d(\Lien) \subseteq \Lien$ for every two-sided ideal $\Lien$ of $\Lieg$.
\item[(b)] For $d \in \Der_{cz}^{\Lie}(\Lieg)$ there exists an $x \in Z^l(\Lieg)$ such that $d_{\mid \gamma_2^{\Lie}(\Lieg)} = R_{x \mid \gamma_2^{\Lie}(\Lieg)}$.
\item[(c)] If $\Lieg$ is 2-step $\Lie$-nilpotent, then $\Der_{cz}^{\Lie}(\Lieg) = {\sf Der}_c^{\Lie}(\Lieg)$.
\item[(d)] If $Z_{\Lie}(\Lieg) =0$, then $\Der_{cz}^{\Lie}(\Lieg) \subseteq {\sf R}(\Lieg)$ and ${\sf R}(Z^l(\Lieg)) \subseteq \Der_{cz}^{\Lie}(\Lieg)$.
\item[(e)] If $\Lieg$ is $\Lie$-nilpotent, then ${\sf Der}_c^{\Lie}(\Lieg)$ is $\Lie$-nilpotent and all $d \in {\sf Der}_c^{\Lie}(\Lieg)$ are nilpotent.
\item[(f)] ${\sf Der}_c^{\Lie}(\Lieg \oplus \Lieg') = {\sf Der}_c^{\Lie}(\Lieg) \oplus {\sf Der}_c^{\Lie}(\Lieg')$, for any Leibniz algebras $\Lieg$ and $\Lieg'$.
\end{enumerate}
\end{Pro}
\begin{proof}
{\it (a)} For any $x \in \Lieg$, $d(x) \in [x, \Lieg]_{\Lie} \subseteq [\Lieg, \Lieg]_{\Lie}$; if $x \in Z_{\Lie}(\Lieg)$, then $d(x) = [x, y]_{lie} = 0$, for all $y \in \Lieg$;   $d(\Lien) \subseteq [\Lien, \Lieg]_{\Lie} \subseteq \Lien$.

{\it (b)} Let $d \in \Der_{cz}^{\Lie}(\Lieg)$, then there exists $x \in Z^l(\Lieg)$ such that $(d-R_x)(\Lieg) \subseteq Z_{\Lie}(\Lieg)$. Since $d-R_x$ is a {\Lie}-derivation we have
$$(d-R_x)([y,z]_{lie}) = [(d-R_x)(y),z]_{lie}+ [y, (d-R_x)(z)]_{lie} = 0$$
and thus $d([y, z]_{lie}) = R_x([y,z]_{lie})$, for all $y, z \in \Lieg$. Hence $d_{\mid \gamma_2^{\Lie}(\Lieg)}=  {R_x}_{\mid \gamma_2^{\Lie}(\Lieg)}$.

{\it (c)}  Notice that if $\Lieg$ is 2-step $\Lie$-nilpotent, then $ \gamma_2^{\Lie}(\Lieg)\subseteq Z_{\Lie}(\Lieg).$ So for all $d \in \Der_{c}^{\Lie}(\Lieg),$  any $x \in Z^l(\Lieg)$ and  $y\in\Lieg,$ we have
$d(y)\in [y,\Lieg]_{\Lie}\subseteq\gamma_2^{\Lie}(\Lieg)\subseteq Z_{\Lie}(\Lieg)$ and $R_x(y)=[y,x]=[y,x]_{lie}\in  \gamma_2^{\Lie}(\Lieg)\subseteq Z_{\Lie}(\Lieg)$.  Therefore $(d-R_x)(\Lieg) \subseteq Z_{\Lie}(\Lieg),$ and thus $d \in \Der_{cz}^{\Lie}(\Lieg).$

{\it (d)} Assume that $Z_{\Lie}(\Lieg) =0.$ Then   for all $d \in \Der_{cz}^{\Lie}(\Lieg),$ there exists an $x \in Z^l(\Lieg)$ such that $(d-R_x)(\Lieg) =0,$ i.e. $d=R_x\in {\sf R}(\Lieg).$ So $\Der_{cz}^{\Lie}(\Lieg) \subseteq{\sf R}(\Lieg).$  The second inclusion can be easily checked.

{\it (e)}  If $\Lieg$ is \Lie-nilpotent of class $c$, then $\gamma_{c+1}^{\Lie}(\Lieg)=0.$ So for any $d\in  {\sf Der}_c^{\Lie}(\Lieg),$ $d(x)\in[x,\Lieg]_{\Lie}\subseteq \gamma_{2}^{\Lie}(\Lieg).$ One inductively proves that $d^c(x) \subseteq \gamma_{c+1}^{\Lie}(\Lieg)$,  $d^{c}(x)=d(d^{c-1}(x))\in[d^{c-1}(x),\Lieg]_{\Lie}\subseteq \gamma_{c+1}^{\Lie}(\Lieg)=0.$ So $d$ is nilpotent.

Also, a routine inductive argument shows that $\gamma_{c+1}^{\Lie}({\sf Der}_c^{\Lie}(\Lieg))(\Lieg)\subseteq \gamma_{c+1}^{\Lie}(\Lieg)=0.$ So $\gamma_{c+1}^{\Lie}({\sf Der}_c^{\Lie}(\Lieg))=0$ and thus ${\sf Der}_c^{\Lie}(\Lieg)$ is $\Lie$-nilpotent.

{\it (f)} For  any $d\in {\sf Der}_c^{\Lie}(\Lieg \oplus \Lieg'),$ it is clear that $d_{\mid\Lieg} \in {\sf Der}_c^{\Lie}(\Lieg) $ and $d_{\mid\Lieg'} \in {\sf Der}_c^{\Lie}(\Lieg').$ Conversely, for $d\in{\sf Der}_c^{\Lie}(\Lieg) $ and  $d'\in{\sf Der}_c^{\Lie}(\Lieg'),$ one easily shows that the mapping  $d'':\Lieg \oplus \Lieg'\to\Lieg \oplus \Lieg'$ defined by $d''(x, x'):= (d(x), d'(x'))$ is a {\Lie}-derivation
such that for $x, x'\in \Lieg,\Lieg'$ we have $d''(x, x')= (d(x), d'(x'))\in ([x,\Lieg]_{\Lie}, [x',\Lieg']_{\Lie})=[(x, x'),\Lieg \oplus \Lieg']_{\Lie}$ by definition of the bracket of $\Lieg \oplus \Lieg'.$
\end{proof}


\section*{Acknowledgements}

Second author was supported by  Agencia Estatal de Investigación (Spain), grant MTM2016-79661-P (AEI/FEDER, UE, support included).


\end{document}